\documentclass[a4paper,reqno,12pt]{marticle}
\usepackage{amsmath,mamsthm,amssymb} % amsthm
\usepackage{amsfonts}
\usepackage[cp866]{inputenc}
\usepackage[english,russian]{babel}
\usepackage{amsbib}
\newtheoremstyle{theorem}%name
{5pt} % space above
{5pt} % space below
{\sl} % bofy font
{\parindent} % ident - empty=no indent, \parindent= paragraph indent
{\bf} % thm head font
{. } % punctuation after thm head
{ } % space after thm head: `` ``=normal \newline=linebreak
{} % thm head specification
\theoremstyle{theorem}
\newtheorem{theorem}{Theorem}
\theoremstyle{theorem}
\newtheorem{corollary}[theorem]{Corollary}
\newtheorem{proposition}[theorem]{Proposition}
\newtheorem{lemma}[theorem]{Lemma}

\newtheoremstyle{defi}%name
{5pt} % space above
{5pt} % space below
{\rm} % bofy font
{\parindent} % ident - empty=no indent, \parindent= paragraph indent
{\bf} % thm head font
{. } % punctuation after thm head
{ } % space after thm head: `` ``=normal \newline=linebreak
{} % thm head specification
\theoremstyle{defi}
\newtheorem{definition}[theorem]{Definition}
\def\proofname{\indent {\sl Proof.}}

\theoremstyle{defi}
\newtheorem{remark}[theorem]{Remark}

\begin{document}
\centerline{\hfill}
\centerline{\hfill}
\centerline{\hfill}
%\centerline{\hfill}
%\centerline{\hfill}                        
			%title
\begin{center}
\bf
SOME PROPERTIES  OF THE RESOLVENT KERNELS \\ FOR CONTINUOUS BI-CARLEMAN
KERNELS
%A NOTE ON FINDING
\end{center}
%\centerline{\hfill}
%\centerline{\hfill}
\vskip.5cm
                         %author
\centerline{Igor M. Novitskii}
                         %address
\begin{center}
Khabarovsk Division\\
Institute of Applied Mathematics \\
Far-Eastern Branch of the Russian Academy of Sciences \\
54, Dzerzhinskiy Street, Khabarovsk 680 000, RUSSIA\\
e-mail: novim@iam.khv.ru
\end{center}
                                %abstract
\noindent{\bf Abstract:} 
We prove that, at regular values lying in a strong convergence region, the
resolvent kernels for a continuous bi-Carleman kernel vanishing at infinity
can be expressed as uniform limits of sequences of resolvent kernels for its
approximating subkernels of Hilbert-Schmidt type.
\vskip\baselineskip
\noindent {\bf AMS Subject Classification:} 45P05, 45A05
\par
\noindent{\bf Key Words:} linear integral equations of the second kind,
                          bounded integral linear operator,
                          Fredholm resolvent, 
                          resolvent kernel,
                          bi-Carleman kernel,
                          Hilbert-Schmidt kernel,
                          regular value,
                          characteristic set
\section{Introduction}

In the general theory of integral equations of the second kind in $L^2(\mathbb{R})$, that
is, equations of the form
\begin{equation}\label{skequ}
f(s)-\lambda\int_{\mathbb{R}}\boldsymbol{T}(s,t)f(t)\,dt =g(s)\quad
\text{for almost every $s\in\mathbb{R}$},
\end{equation}
it is customary to call an integral kernel $\boldsymbol{T}_{\mid\lambda}$
a \textit{resolvent kernel for $\boldsymbol{T}$ at $\lambda$} if the integral
operator it induces on $L^2(\mathbb{R})$ is the Fredholm resolvent
$T(I-\lambda T)^{-1}$ of the integral operator $T$ on $L^2(\mathbb{R})$, whose
kernel is $\boldsymbol{T}$. Once the resolvent kernel $\boldsymbol{T}_{\mid\lambda}$
has been constructed, one can express the $L^2(\mathbb{R})$ solution $f$ to
equation \eqref{skequ} in a direct and simple fashion as
\begin{equation*}\label{solution}
            f(s)=g(s)+\lambda\int_{\mathbb{R}}\boldsymbol{T}_{\mid\lambda}(s,t)g(t)\,dt
            \quad\text{for almost every $s\in\mathbb{R}$},
\end{equation*}
regardless of the particular choice of the function $g$ of $L^2(\mathbb{R})$.
Here it should be noted that, in general, the property of being an integral
operator is not shared by Fredholm resolvents of integral operators, and there
is even an example, given in \cite{Kor:nonint1} (see also \cite[Section~5, Theorem~8]{Kor:alg}),
of an integral operator whose Fredholm resolvent at any non-zero regular value
is not an integral operator. This phenomenon, however, can never occur for
Carleman operators due to the fact that the right-multiple by a bounded
operator of a Carleman operator is again a Carleman operator. Therefore, in
the case when the kernel $\boldsymbol{T}$ is Carleman and $\lambda$ is a
regular value for $T$, the problem of solving equation \eqref{skequ} may be
reduced to the problem of explicitly constructing in terms of $\boldsymbol{T}$
the resolvent kernel $\boldsymbol{T}_{\mid\lambda}$ which is a priori known
to exist. For a precise formulation of this latter problem and for comments
to the solution of some of its special cases we refer to the works by Korotkov
\cite{Kor:problems}, \cite{Kor:alg} (in both the references, see Problem~4 in
\S 5). Here we only notice that in the case when a measurable kernel
$\boldsymbol{T}$ of \eqref{skequ} is bi-Carleman but otherwise unrestricted,
there seems to be as yet no analytic machinery for explicitly constructing its
resolvent kernel $\boldsymbol{T}_{\mid\lambda}$ at every regular value
$\lambda$. In order to approach this problem, which motivates the present work, we
confine our investigation to the case in which the kernel
$\boldsymbol{T}:\mathbb{R}^2\to\mathbb{C}$ of \eqref{skequ} and its two
Carleman functions $\boldsymbol{t}(s)=\overline{\boldsymbol{T}(s,\cdot)}$,
$\boldsymbol{t}^{\boldsymbol{\prime}}(s)
=\boldsymbol{T}(\cdot,s):\mathbb{R}\to L^2(\mathbb{R})$
are continuous and vanish at infinity.
These conditions can always be achieved by means of a unitary equivalence
transformation (see Proposition~\ref{uneqv} below), and this is, therefore,
not a serious loss of generality when working in the class of such kernels
(called $K^0$-kernels). One of the main technical advantages of dealing with
a $K^0$-kernel is that its subkernels, such as the restrictions of it to
compact squares in $\mathbb{R}^2$ centered at origin, are quite amenable to
the methods of the classical theory of ordinary integral equations, and can be
used to approximate the original kernel in suitable norms.
This, for instance, can be used directly to establish an explicit theory
of spectral functions for any Hermitian $K^0$-kernel by a development
essentially the same as the one given by T.~Carleman: for a symmetric Carleman kernel
that is the pointwise limit of its symmetric Hilbert-Schmidt subkernels
satisfying a mean square continuity condition, he constructed in
\cite[pp.~25-51]{Carl:book} its spectral functions as pointwise limits of
sequences of spectral functions for the subkernels. (For further developments and applications of
Carleman's spectral theory we refer to \cite{Neu}, \cite{Trji}, \cite{Akh}, 
\cite[Appendix~I]{Akh:Glaz}, \cite{Costley}, \cite{Will}, and \cite{Kor:book1}.)
\par
Following this subkernel approach we focus in present paper on the question
whether and at what regular values $\lambda$ the resolvent kernel
$\boldsymbol{T}_{\mid\lambda}$ for a $K^0$-kernel $\boldsymbol{T}$
can be expressed as the limit of a sequence of resolvent kernels for 
the subkernels of $\boldsymbol{T}$. The main result of the paper is
Theorem~\ref{thmeq3.1} describing such regular values $\lambda$ in terms
of generalized strong convergence, introduced by T.~Kato in \cite{Kato:book}.
\section{Notation, Definitions, and Auxiliary Facts}
\subsection{Fredholm Resolvents and Characteristic Sets}\label{freres}
Throughout this paper, the symbols $\mathbb{C}$ and $\mathbb{N}$ refer to the
complex plane and the set of all positive integers, respectively, $\mathbb{R}$
is the real line equipped with the Lebesgue measure, and $L^2=L^2(\mathbb{R})$
is the complex Hilbert space of (equivalence classes of) measurable
complex-valued functions on $\mathbb{R}$ equipped with the inner product
$\langle f,g\rangle=\int f(s)\overline{g(s)}\,ds$ and the norm
$\left\|f\right\|=\langle f,f\rangle^{\frac{1}2}$. (Integrals with no indicated
domain, such as the above, are always to be extended over $\mathbb{R}$.) If
$L\subset L^2$, we write $\overline{L}$ for the norm closure of $L$ in $L^2$,
$L^{\perp}$ for the orthogonal complement of $L$ in $L^2$, and
$\mathrm{Span}(L)$ for the norm closure of the set of all linear combinations
of elements of $L$. Recall that a set $L$ in a normed space $Y$ is said to be
\textit{relatively compact} in $Y$ if each sequence of elements from $L$
contains a subsequence converging in the norm of $Y$.
\par
Let $\mathfrak{R}\left(L^2\right)$ denote the Banach algebra of all bounded
linear operators acting on $L^2$; $\|\cdot\|$  will also denote the norm in
$\mathfrak{R}\left(L^2\right)$. For an operator $A$ of $\mathfrak{R}(L^2)$,
$A^*$ stands for the adjoint to $A$ with respect to
$\langle \cdot,\cdot\rangle$, $\mathrm{Ran\,}A=\left\{Af\mid f\in L^2\right\}$
for the range of $A$, and $\mathrm{Ker\,} A=\left\{f\in L^2\mid Af=0\right\}$
for the null-space of $A$. An operator $U\in\mathfrak{R}(L^2)$ is said to be
\textit{unitary} if $\mathrm{Ran\,}U=L^2$ and
$\langle Uf,Ug\rangle=\langle f,g\rangle$ for all $f$, $g\in L^2$. An operator
$A\in\mathfrak{R}(L^2)$ is said to be \textit{invertible} if it has an inverse
which is also in $\mathfrak{R}(L^2)$, that is, if there is an operator
$B\in\mathfrak{R}(L^2)$ for which $BA=AB=I$, where $I$ is the identity operator
on $L^2$; $B$ is denoted by $A^{-1}$. An operator $P\in\mathfrak{R}(L^2)$ is
called a \textit{projection} in $L^2$ if $P^2=P$, and a projection $P$ in $L^2$
is said to be \textit{orthogonal} if $P=P^*$. An operator $T\in\mathfrak{R}(L^2)$
is said to be \textit{compact} if it transforms every bounded set in $L^2$ into
a relatively compact set in $L^2$. A (compact) operator $A\in\mathfrak{R}(L^2)$
is \textit{nuclear} if
$\sum_n\left|\left\langle Au_n,u_n\right\rangle\right|<\infty$
for any choice of an orthonormal basis $\{u_n\}$ of $L^2$.
\par
Throughout the rest of this subsection, $T$ denotes a bounded linear operator
of $\mathfrak{R}\left(L^2\right)$. The set of \textit{regular values} for $T$,
denoted by $\Pi(T)$, is the set of complex numbers $\lambda$ such that the
operator $I-\lambda T$ is invertible, that is, it has an inverse
$R_\lambda(T)=\left(I-\lambda T\right)^{-1}
\in\mathfrak{R}\left(L^2\right)$ that satisfies
\begin{equation}\label{eqress}
\left(I-\lambda T\right)R_\lambda(T)=R_\lambda(T)\left(I-\lambda T\right)=I.
\end{equation}
The operator
\begin{equation}\label{eqresf}
T_{\mid\lambda}:= TR_\lambda(T)\ (=R_\lambda(T)T)
\end{equation}
is then referred to as the \textit{Fredholm resolvent} of $T$ at $\lambda$.
Remark that if $\lambda$ is a regular value for $T$, then, for each fixed
$g$ in $L^2$, the (unique) solution $f$ of $L^2$ to the second-kind equation
$f-\lambda Tf=g$ may be written as
\begin{equation}\label{equnsol}
f=g+\lambda T_{\mid\lambda}g
\end{equation}
(follows from the formula
\begin{equation}\label{resFres}
R_\lambda(T)=I+\lambda T_{\mid\lambda}
\end{equation}
which is a rewrite of \eqref{eqress}). Recall that the inverse 
$R_\lambda(T)$ of $I-\lambda T$ as a function of $T$ also satisfies
the following identity, often referred to as the
\textit{second resolvent equation} (see, e.g., \cite[Theorem~5.16.1]{Hille}):
for $T$, $A\in\mathfrak{R}\left(L^2\right)$,
\begin{equation}
\begin{split}\label{secResEq}
R_\lambda(T)-R_\lambda(A)&=\lambda R_\lambda(T)(T-A)R_\lambda(A)\\&=
\lambda R_\lambda(A)(T-A)R_\lambda(T)\quad\text{for every $\lambda\in\Pi(T)\cap\Pi(A)$.}
\end{split}
\end{equation}
(A slightly modified version of it is
\begin{equation}
\begin{split}\label{secFResEq}
T_{\mid\lambda}-A_{\mid\lambda}
&=(I+\lambda T_{\mid\lambda})(T-A)(I+\lambda A_{\mid\lambda})\\&=
(I+\lambda A_{\mid\lambda})(T-A)(I+\lambda T_{\mid\lambda}) \quad\text{for every $\lambda\in\Pi(T)\cap\Pi(A)$,}
\end{split}
\end{equation}
which involves the Fredholm resolvents.)
It should also be mentioned that the map $R_\lambda(T)\colon\Pi(T)\to\mathfrak{R}\left(L^2\right)$
(resp., $T_\lambda\colon\Pi(T)\to\mathfrak{R}\left(L^2\right)$) is
continuous at every point $\lambda$ of the open set $\Pi(T)$, in the sense
that $\|R_{\lambda_n}(T)-R_{\lambda}(T)\|\to0$ (resp.,
$\|T_{\lambda_n}-T_{\lambda}\|\to0$) when $\lambda_n\to\lambda$,
$\lambda_n\in\Pi(T)$ (see, e.g., \cite[Lemma~2 (XIII.4.3)]{KanAk}).
Moreover, $R_\lambda(T)$ is given by an operator-norm convergent series ($T^0=I$):
\begin{equation}\label{resseries}
R_\lambda(T)=\sum_{n=0}^\infty\lambda^nT^n
\quad\text{provided $|\lambda|<r(T):=\frac1{\lim\limits_{n\to\infty}\sqrt[n]{\|T^n\|}}$}
\end{equation}
(see, e.g., \cite[Theorem~1 (XIII.4.2)]{KanAk}).
For notational simplicity, we shall always write $R_\lambda^*(T)$ for the
adjoint $\left(R_\lambda(T)\right)^*$ to $R_\lambda(T)$.
\par
The \textit{characteristic set $\Lambda(T)$} for $T$ is defined to be the
complementary set in $\mathbb{C}$ of $\Pi(T)$: $\Lambda(T)=\mathbb{C}\setminus\Pi(T)$.
\par
Given a sequence $\left\{S_n\right\}_{n=1}^\infty$ of bounded operators on
$L^2$, let $\nabla_\mathfrak{b}(\{S_n\})$ denote the set of all nonzero
complex numbers $\zeta$ for which there exist positive constants $M(\zeta)$
and $N(\zeta)$ such that
\begin{equation}\label{eq3.2}
\zeta\in\Pi(S_n)\ \text{and}\ \left\|S_{n\mid\zeta}\right\|\leqslant M(\zeta)
\quad\text{for $n>N(\zeta)$,}
\end{equation}
where, as in what follows,  $S_{n\mid\zeta}$ stands for the Fredholm resolvent
of $S_n$ at $\zeta$, and let $\nabla_\mathfrak{s}(\{S_n\})$ denote the set of
all nonzero complex numbers $\zeta$ ($\in\nabla_\mathfrak{b}(\{S_n\})$)
for which the sequence $\left\{S_{n\mid\zeta}\right\}$ is convergent
in the strong operator topology (that is to say, the limit
$\lim\limits_{n\to\infty}S_{n\mid\zeta}f$ exists in $L^2$ for every $f\in L^2$).
\begin{remark}\label{1zremark}
The set $\nabla_\mathfrak{b}(\{S_n\})$ (resp. $\nabla_\mathfrak{s}(\{S_n\})$)
evidently remains unchanged if in its definition the Fredholm resolvents
$S_{n\mid\zeta}$ are replaced by the operators
$R_\zeta(S_n)=(I-\zeta S_n)^{-1}=I+\zeta S_{n\mid\zeta}$ (cf. \eqref{resFres}).
So, if $\Delta_\mathrm{b}$ (resp., $\Delta_\mathrm{s}$) is the \textit{region of
boundedness} (resp., \textit{strong convergence}) for the resolvents
$\left\{(\zeta I-S_n)^{-1}\right\}$,
which was introduced and studied in \cite[Section VIII-1.1]{Kato:book}, then
the sets $\nabla_\mathfrak{b}(\{S_n\})$ and $\Delta_\mathrm{b}\setminus\{0\}$
(resp., $\nabla_\mathfrak{s}(\{S_n\})$ and $\Delta_\mathrm{s}\setminus\{0\}$)
are mapped onto each other by the mapping $\zeta\to\zeta^{-1}$.
In the course of the proof of Theorem~\ref{thmeq3.1} below, this mapping is
always kept in mind when referring to \cite{Kato:book} for generalized strong
convergence theory.
\end{remark}
\subsection{Integral Operators}
A linear operator $T:L^2\to L^2$ is \textit{integral} if there is a
complex-valued measurable function $\boldsymbol{T}$ (\textit{kernel}) on
$\mathbb{R}^2$ such that
\begin{equation*}
               (Tf)(s)=\int\boldsymbol{T}(s,t)f(t)\,dt
\end{equation*}
for every $f\in L^2$ and almost every $s\in\mathbb{R}$. Recall
\cite[Theorem 3.10]{Halmos:Sun} that integral operators are bounded, and need
not be compact. A measurable function $\boldsymbol{T}:\mathbb{R}^2\to\mathbb{C}$
is said to be a \textit{Carleman kernel} if $\boldsymbol{T}(s,\cdot)\in L^2$ for almost every fixed
$s$ in $\mathbb{R}$. To each Carleman kernel $\boldsymbol{T}$ there corresponds
a \textit{Carleman function} $\boldsymbol{t}:\mathbb{R}\to L^2$ defined by
$\boldsymbol{t}(s)=\overline{\boldsymbol{T}(s,\cdot)}$ for all $s$ in
$\mathbb{R}$ for which $\boldsymbol{T}(s,\cdot)\in L^2$. The Carleman kernel
$\boldsymbol{T}$ is called \textit{bi-Carleman} in case its conjugate
transpose kernel $\boldsymbol{T}^{\boldsymbol{\prime}}$
($\boldsymbol{T}^{\boldsymbol{\prime}}(s,t)=\overline{\boldsymbol{T}(t,s)}$)
is also a Carleman kernel. Associated with the conjugate transpose
$\boldsymbol{T}^{\boldsymbol{\prime}}$ of every bi-Carleman kernel
$\boldsymbol{T}$ there is therefore a Carleman function
$\boldsymbol{t}^{\boldsymbol{\prime}}:\mathbb{R}\to L^2$ defined by
$\boldsymbol{t}^{\boldsymbol{\prime}}(s)
=\overline{\boldsymbol{T}^{\boldsymbol{\prime}}(s,\cdot)}
\left(=\boldsymbol{T}(\cdot,s)\right)$ for all $s$ in $\mathbb{R}$
for which $\boldsymbol{T}^{\boldsymbol{\prime}}(s,\cdot)\in L^2$.
With each bi-Carleman kernel $\boldsymbol{T}$, we therefore associate the pair
of Carleman functions $\boldsymbol{t}$,
$\boldsymbol{t}^{\boldsymbol{\prime}}:\mathbb{R}\to L^2$, both defined, via
$\boldsymbol{T}$, as above. An integral operator whose kernel is Carleman
(resp., bi-Carleman) is referred to as the \textit{Carleman}
(resp., \textit{bi-Carleman}) operator. The integral operator $T$ is called
\textit{bi-integral} if its adjoint $T^*$ is also an integral operator;
in that case if $\boldsymbol{T}^{\boldsymbol{\ast}}$ is the kernel of $T^*$
then, in the above notation,
$\boldsymbol{T}^{\boldsymbol{\ast}}(s,t)=\boldsymbol{T}^{\boldsymbol{\prime}}(s,t)$
for almost all $(s,t)\in\mathbb{R}^2$ (see, e.g., \cite[Theorem 7.5]{Halmos:Sun}).
A bi-Carleman operator is always a bi-integral operator, but not conversely.
The bi-integral operators are generally involved in second-kind integral
equations (like \eqref{skequ}) in $L^2$, as the adjoint equations to such
equations are customarily required to be integral.
A kernel $\boldsymbol{T}$ on $\mathbb{R}^2$ is said to be
\textit{Hilbert-Schmidt} if
$\int\int|\boldsymbol{T}(s,t)|^2\,dt\,ds<\infty.$
A nuclear operator on $L^2$ is always an integral operator, whose kernel is Hilbert-Schmidt (see, e.g., \cite{Reed:Sim1}).
We shall employ the convention of referring to integral operators by
italic caps and to the corresponding kernels (resp., Carleman functions)
by the same letter, but written in upper case (resp., lower case) bold-face type.
Thus, e.g., if $T$ denotes, say, a bi-Carleman operator, then
$\boldsymbol{T}$ and $\boldsymbol{t}$, $\boldsymbol{t}^{\boldsymbol{\prime}}$
are to be used to denote its kernel and two Carleman functions, respectively.
\par
We conclude this subsection by recalling an important algebraic property of Carleman operators
which will be exploited frequently throughout the text, a property
which is the content of the following so-called ``Right-Multipilication Lemma''
(cf. \cite{Misra}, \cite[Corollary IV.2.8]{Kor:book1}, or
\cite[Theorem 11.6]{Halmos:Sun}):
\begin{proposition}\label{rimlt}
Let $T$ be a Carleman operator, let $\boldsymbol{t}$ be the Carleman function 
associated with the inducing Carleman kernel of $T$, and let
$A\in\mathfrak{R}\left(L^2\right)$ be arbitrary. Then the product operator $TA$
is also a Carleman operator, and the composition function
\begin{equation}\label{frimlt}
A^*(\boldsymbol{t}(\cdot)):\mathbb{R}\to L^2
\end{equation}
is the Carleman function associated with its kernel.
\end{proposition}
\subsection{$\boldsymbol{K^0}$-Kernels}
If $k$ is in $\mathbb{N}$ and $B$ is a Banach space with norm $\left\|\cdot\right\|_B$, let
$C(\mathbb{R}^k,B)$ denote the Banach space, with the norm
$\left\|f\right\|_{C(\mathbb{R}^k,B)}=\sup\limits_{x\in \mathbb{R}^k}\left\|f(x)\right\|_B$,
of all continuous functions $f$ from $\mathbb{R}^k$ into $B$ such that
$\lim\limits_{|x|\to\infty}\|f(x)\|_B=0$,
where $|\cdot|$ is the euclidian norm in $\mathbb{R}^k$.
Given an equivalence class $f\in L^2$ containing a function of
$C(\mathbb{R},\mathbb{C})$, the symbol $[f]$ is used to mean that
function.
\begin{definition} \label{def:ker}
A bi-Carleman kernel $\boldsymbol{T}\colon\mathbb{R}^2\to\mathbb{C}$ is called 
a $K^0$-\textit{kernel} if the following three conditions are satisfied:
\par
(i) the function $\boldsymbol{T}$ is in $C\left(\mathbb{R}^2,\mathbb{C}\right)$,
\par
(ii) the Carleman function $\boldsymbol{t}$ associated with $\boldsymbol{T}$,
$\boldsymbol{t}(s)=\overline{\boldsymbol{T}(s,\cdot)}$, is in
$C\left(\mathbb{R},L^2\right)$,
\par
(iii) the Carleman function $\boldsymbol{t}^{\boldsymbol{\prime}}$ associated
with the conjugate transpose $\boldsymbol{T}^{\boldsymbol{\prime}}$ of
$\boldsymbol{T}$, $\boldsymbol{t}^{\boldsymbol{\prime}}(s)
=\overline{\boldsymbol{T}^{\boldsymbol{\prime}}(s,\cdot)}=
\boldsymbol{T}(\cdot,s)$, is in
$C\left(\mathbb{R},L^2\right)$.
\end{definition}
What follows is a brief discussion of some properties of $K^0$-kernels
relevant for this paper. In the first place, note that  the conditions figuring in Definition~\ref{def:ker}
do not depend on each other in general; it is therefore natural to discuss the
role played by each of them separately. The more restrictive of these conditions 
is (i), in the sense that it rules out the possibility for any $K^0$-kernel
(unless that kernel is identically zero) of being a function depending only on 
the sum, difference, or product of the variables;
there are many other less trivial examples of inadmissible dependences.
This circumstance may be of use in constructing examples of those bi-Carleman 
kernels that have both the properties (ii) and (iii), but do not enjoy (i); 
for another reason of existence of such type bi-Carleman kernels, we refer to 
a general remark in \cite[p.~115]{Zaanen} also concerning compactly supported kernels. In this connection, it can, however,
be asserted that if a function $\boldsymbol{T}\in C\left(\mathbb{R}^2,\mathbb{C}\right)$
additionally satisfies $\left|\boldsymbol{T}(s,t)\right|\leqslant p(s)q(t)$, with
$p$, $q$ being $C(\mathbb{R},\mathbb{R})$ functions square integrable over
$\mathbb{R}$, then $\boldsymbol{T}$ is a $K^0$-kernel, that is to say, the
Carleman functions $\boldsymbol{t}$, $\boldsymbol{t}^{\boldsymbol{\prime}}$ it 
induces are both in $C\left(\mathbb{R},L^2\right)$. The assertion may be proved
by an extension from the positive definite case with
$p(s)\equiv q(s)\equiv (\boldsymbol{T}(s,s))^\frac{1}{2}$ to this general case of Buescu's
argument in \cite[pp. 247--249]{Buescu1}.
\par
A few remarks are in order here concerning what can immediately be inferred
from the $C\left(\mathbb{R},L^2\right)$-behaviour of the Carleman functions
$\boldsymbol{t}$, $\boldsymbol{t}^{\boldsymbol{\prime}}$
associated with a given $K^0$-kernel $\boldsymbol{T}$
(thought of as a kernel of an integral operator $T\in\mathfrak{R}\left(L^2\right)$):
\par
1) The images of $\mathbb{R}$ under $\boldsymbol{t}$, 
$\boldsymbol{t}^{\boldsymbol{\prime}}$, that is,
\begin{equation}\label{precom}
\boldsymbol{t}(\mathbb{R}):=\bigcup\limits_{s\in\mathbb{R}}\boldsymbol{t}(s),
\quad\boldsymbol{t}^{\boldsymbol{\prime}}(\mathbb{R})
:=\bigcup\limits_{s\in\mathbb{R}}\boldsymbol{t}^{\boldsymbol{\prime}}(s),
\end{equation}
are relatively compact sets in $L^2$;
\par
2) The \textit{Carleman norm-functions} $\boldsymbol{\tau}$ and $\boldsymbol{\tau}^{\boldsymbol{\prime}}$,
defined on $\mathbb{R}$ by
$\boldsymbol{\tau}(s)=\left\|\boldsymbol{t}(s)\right\|$ and
$\boldsymbol{\tau}^{\boldsymbol{\prime}}(s)=\left\|\boldsymbol{t}^{\boldsymbol{\prime}}(s)\right\|$, respectively,
are continuous vanishing at infinity, that is to say,
\begin{equation}\label{eqnormkf}
\boldsymbol{\tau}, \boldsymbol{\tau}^{\boldsymbol{\prime}}\in C(\mathbb{R},\mathbb{R}).
\end{equation}
\par
3) The images $Tf$ and $T^*f$ of any $f\in L^2$ under $T$ and $T^*$, respectively, 
have $C(\mathbb{R},\mathbb{C})$-representatives in $L^2$, $[Tf]$ and $[T^*f]$,
defined pointwise on $\mathbb{R}$ as
\begin{equation}\label{TfT*f}
[Tf](s)=\langle f,\boldsymbol{t}(s)\rangle,
\quad
[T^*f](s)=\langle f,\boldsymbol{t}^{\boldsymbol{\prime}}(s)\rangle
\quad\text{for every $s$ in $\mathbb{R}$}.
\end{equation}
\par
4) Using \eqref{TfT*f}, it is easy to deduce that
$\boldsymbol{t}(\mathbb{R})^\perp=\mathrm{Ker\,}T$,
$\boldsymbol{t}^{\boldsymbol{\prime}}(\mathbb{R})^\perp=\mathrm{Ker\,}T^*$.
(Indeed:
\begin{gather*}
f\in \boldsymbol{t}(\mathbb{R})^\perp\iff\langle f,\boldsymbol{t}(s)\rangle=0\
\forall s\in\mathbb{R}\iff f\in \mathrm{Ker\,}T,
\\
f\in \boldsymbol{t}^{\boldsymbol{\prime}}(\mathbb{R})^\perp
\iff\langle f,\boldsymbol{t}^{\boldsymbol{\prime}}(s)\rangle\ \forall s\in\mathbb{R}\iff f\in \mathrm{Ker\,}T^*.)
\end{gather*}
The orthogonality between the range of an operator and the null-space of
its adjoint then yields
\begin{gather*}
\mathrm{Span\,}(\boldsymbol{t}(\mathbb{R}))=
\left(\boldsymbol{t}(\mathbb{R})^\perp\right)^\perp
=
\overline{\mathrm{Ran\,}T^*},
\\
\mathrm{Span\,}\left(\boldsymbol{t}^{\boldsymbol{\prime}}(\mathbb{R})\right)
=\left(\boldsymbol{t}^{\boldsymbol{\prime}}(\mathbb{R})^\perp\right)^\perp=
\overline{\mathrm{Ran\,}T}.
\end{gather*}
\par
5) The $n$-th iterant $\boldsymbol{T}^{[n]}$ ($n\geqslant2$) of the $K^0$-kernel
$\boldsymbol{T}$,
\begin{equation}\label{iterant}
\boldsymbol{T}^{[n]}(s,t):=\underbrace{\int\dots\int}_{n-1}\boldsymbol{T}(s,\xi_{1})\,\dots\,
\boldsymbol{T}(\xi_{n-1},t)\,d\xi_{1}\,\dots\,d\xi_{n-1}
\left(=\langle T^{n-2}\left(\boldsymbol{t}^{\boldsymbol{\prime}}(t)\right),
\boldsymbol{t}(s)\rangle\right),
\end{equation}
is a $K^0$-kernel that defines the integral operator $T^n$.
More generally, every two $K^0$-kernels $\boldsymbol{P}$, $\boldsymbol{Q}$
might be said to be \textit{multipliable} with each other, in the sense that 
their convolution
\begin{equation*}
\boldsymbol{C}(s,t):=\int\boldsymbol{P}(s,\xi)\boldsymbol{Q}(\xi,t)\,d\xi
\left(=\langle\boldsymbol{q}^{\boldsymbol{\prime}}(t),\boldsymbol{p}(s)\rangle\right)
\end{equation*}
exists at every point $(s,t)\in\mathbb{R}^2$, and forms a $K^0$-kernel
that defines the product operator $C=PQ$:
\begin{equation}\label{fkerPQ}
\begin{gathered}
\int\left\langle\boldsymbol{q}^{\boldsymbol{\prime}}(t),
\boldsymbol{p}(s)\right\rangle h(t)\,dt
=\int\left(\int\boldsymbol{P}(s,\xi)
\boldsymbol{Q}(\xi,t)\,d\xi\right)h(t)\,dt
=\left\langle h,Q^*(\boldsymbol{p}(s))\right\rangle
\\
=
\left\langle Qh,\boldsymbol{p}(s)\right\rangle
=
\int\boldsymbol{P}(s,\xi)
\left(\int\boldsymbol{Q}(\xi,t)h(t)\,dt\right)\,d\xi
=\left[PQh\right](s).
\end{gathered}
\end{equation}
Since both $\boldsymbol{p}$ and $\boldsymbol{q}^{\boldsymbol{\prime}}$ are
in $C\left(\mathbb{R},L^2\right)$ and both $P$ and $Q$ are in
$\mathfrak{R}\left(L^2\right)$,
the fact that $\boldsymbol{C}$ satisfies Definition~\ref{def:ker}
may be derived from the joint continuity of the inner product in its two arguments when proving (i),
and from Proposition~\ref{rimlt}, according to which
\begin{equation*}
\boldsymbol{c}(s)=\overline{\boldsymbol{C}(s,\cdot)}=Q^*(\boldsymbol{p}(s)),
\quad
\boldsymbol{c}^{\boldsymbol{\prime}}(s)=\boldsymbol{C}(\cdot,s)=
P\left(\boldsymbol{q}^{\boldsymbol{\prime}}(s)\right) \quad\text{for every $s$ in $\mathbb{R}$,}
\end{equation*}
when proving both (ii) and (iii).
\subsection{Sub-$\boldsymbol{K}^{0}$-Kernels}\label{}
If $\boldsymbol{T}$ is a $K^0$-kernel of an integral operator $T$, then impose
on $\boldsymbol{T}$ an extra condition of being of special parquet support:
\par
(iv) there exist positive reals $\tau_n$ ($n\in\mathbb{N}$) strictly
increasing to $+\infty$ such that, for each fixed $n$, the
\textit{subkernels $\boldsymbol{T}_n$, $\boldsymbol{\widetilde{T}}_n$ of $\boldsymbol{T}$},
defined on $\mathbb{R}^2$ by
\begin{equation}\label{ke1.3.3}
\boldsymbol{T}_n(s,t)=\chi_n(s)\boldsymbol{T}(s,t),\quad
\boldsymbol{\widetilde{T}}_n(s,t)=\boldsymbol{T}_n(s,t)\chi_n(t),
\end{equation}
are $K^0$-kernels, and the integral operators
\begin{equation}\label{TnTn}
T_n:=P_nT,\quad \widetilde{T}_n:=P_nTP_n
\end{equation}
they induce on $L^2$ are nuclear;
here, as in the rest of the paper, $\chi_n$ stands for the characteristic
function of the open interval $\mathbb{I}_n=\left(-\tau_n,\tau_n\right)$,
and $P_n$ for an orthogonal projection defined
on each $f\in L^2$ by $P_nf=\chi_n f$ (so that $(I-P_n)f=\widehat{\chi}_nf$ for each
$f\in L^2$, where $\widehat{\chi}_n$ is the characteristic function of
the set $\widehat{\mathbb{I}}_n:= \mathbb{R}\setminus\mathbb{I}_n$).
\par
Condition (iv) implies that, for each $n$, the kernel $\boldsymbol{T}(s,t)$
does vanish everywhere on the straight lines $s=\pm\tau_n$ and $t=\pm\tau_n$,
parallel to the axes of $t$ and $s$, respectively.
$P_n$ ($n\in\mathbb{N}$) form a sequence of orthogonal projections increasing to $I$
with respect to the strong operator topology, so that, for every $f\in L^2$,
\begin{equation}\label{eqPntoI}
\left\|\left(P_n-I\right)f\right\|\searrow 0\quad \text{as $n\to\infty$}.
\end{equation}
So it follows immediately from \eqref{TnTn} that
\begin{equation}\label{eqTntoT}
\begin{gathered}
\left\|\left({T}_n-T\right)f\right\|\to 0,\quad
\left\|\left(\widetilde{T}_n-T\right)f\right\|\to 0,\\
\left\|\left({T}^*_n-T^*\right)f\right\|\to 0,
\quad\left\|\left(\widetilde{T}^*_n-T^*\right)f\right\|\to 0
\end{gathered}
\end{equation}
as $n$ tends to infinity.
\par
Among the subkernels defined  in \eqref{ke1.3.3}, the $\boldsymbol{T}_n$ have more in
common with the original kernel $\boldsymbol{T}$, as
$[T_nf](s)=\int\boldsymbol{T}(s,t)f(t)\,dt$
for all $s\in\mathbb{I}_n$ and any $f\in L^2$, while
the subkernels $\boldsymbol{\widetilde{T}}_n$ are more suitable to deal
with $\boldsymbol{T}$ being Hermitian, that is, satisfying
$\boldsymbol{T}(s,t)=\overline{\boldsymbol{T}(t,s)}$ for all $s$,
$t\in\mathbb{R}$, because then they all are also Hermitian.
\par
Now we list some basic properties of the subkernels defined in \eqref{ke1.3.3},
most of which are obvious from the definition:
\begin{equation}\label{ke1.3.1in}
\left|\boldsymbol{T}_n(s,t)\right|\leqslant\left|\boldsymbol{T}(s,t)\right|,
\quad\left|\boldsymbol{\widetilde{T}}_n(s,t)\right|\leqslant\left|\boldsymbol{T}(s,t)\right|,
\quad\text{for all $s$, $t\in\mathbb{R}$},
\end{equation}
\begin{equation}\label{ke1.3.1}
\lim_{n\to\infty}\left\|\boldsymbol{T}_n-\boldsymbol{T}\right\|_{C\left(\mathbb{R}^2,\mathbb{C}\right)}=0,\quad
\lim_{n\to\infty}\left\|\boldsymbol{\widetilde{T}}_n-\boldsymbol{T}\right\|_{C\left(\mathbb{R}^2,\mathbb{C}\right)}=0,
\end{equation}
\begin{equation}\label{ke1.3.3gsch}
\int\int\left|\boldsymbol{T}_n(s,t)\right|^2\,dt\,ds<\infty,
\quad\int\int\left|\boldsymbol{\widetilde{T}}_n(s,t)\right|^2\,dt\,ds<\infty,
\end{equation}
\begin{equation}\label{ke1.3.2}
\begin{gathered}
\lim_{n\to\infty}\left\|\boldsymbol{t}_n-\boldsymbol{t}\right\|_{C\left(\mathbb{R},L^2\right)}=0,\quad
\lim_{n\to\infty}\left\|\boldsymbol{t}^{\boldsymbol{\prime}}_n-\boldsymbol{t}^{\boldsymbol{\prime}}\right\|_{C\left(\mathbb{R},L^2\right)}=0,
\\
\lim_{n\to\infty}\left\|\boldsymbol{\widetilde{t}}_n-\boldsymbol{t}\right\|_{C\left(\mathbb{R},L^2\right)}=0,\quad
\lim_{n\to\infty}\left\|\boldsymbol{\widetilde{t}}^{\boldsymbol{\prime}}_n-\boldsymbol{t}^{\boldsymbol{\prime}}\right\|_{C\left(\mathbb{R},L^2\right)}=0,
\end{gathered}
\end{equation}
where
\begin{equation}\label{kesubcf}
\begin{gathered}
\boldsymbol{t}_n(s)=\overline{\boldsymbol{T}_n(s,\cdot)}=\chi_n(s)\boldsymbol{t}(s),
\quad\boldsymbol{t}^{\boldsymbol{\prime}}_n(t)=\boldsymbol{T}_n(\cdot,t)
=P_n\left(\boldsymbol{t}^{\boldsymbol{\prime}}(t)\right),
\\
\boldsymbol{\widetilde{t}}_n(s)=\overline{\boldsymbol{\widetilde{T}}_n(s,\cdot)}=\chi_n(s)P_n\left(\boldsymbol{t}(s)\right),
\quad
\boldsymbol{\widetilde{t}}^{\boldsymbol{\prime}}_n(t)=\boldsymbol{\widetilde{T}}_n(\cdot,t)
=\chi_n(t)P_n\left(\boldsymbol{t}^{\boldsymbol{\prime}}(t)\right)
\end{gathered}
\end{equation}
are the associated Carleman functions.
The limits in \eqref{ke1.3.2} all hold due to (ii), (iii), \eqref{eqPntoI}, and
a result from \cite[Lemma~3.7, p.~151]{Kato:book}. The result, just referred
to, will be used in the text so often that it should be explicitly stated.
\begin{lemma}\label{lemKato}
Let $S_n$, $S\in\mathfrak{R}\left(L^2\right)$, and suppose that,
for any $x\in L^2$,
$\|S_nx-Sx\|\to0$ as $n\to\infty$. Then for any relatively compact set
$U$ in $L^2$
\begin{equation}\label{simKatolem}
\sup_{x\in U}\|S_nx-Sx\|\to0\quad\text{as $n\to\infty$}.
\end{equation}
\end{lemma}
Applying this lemma to the sets $\boldsymbol{t}(\mathbb{R})$ and
$\boldsymbol{t}^{\boldsymbol{\prime}}(\mathbb{R})$ of \eqref{precom}
immediately gives that
\begin{equation}\label{Katolem}
\sup_{s\in\mathbb{R}}\|(S_n-S)(\boldsymbol{t}(s))\|\to0,
\quad
\sup_{t\in\mathbb{R}}\|(S_n-S)(\boldsymbol{t}^{\boldsymbol{\prime}}(t))\|\to0
\quad\text{as $n\to\infty$}.
\end{equation}
We would like to close this section with a unitary equivalence result which is
essentially contained in Theorem~1 of \cite{nov:Lon},  where it is proved for
operators on $L^2[0,+\infty)$ and with the sequence $\{t_n\}$ playing the role
of the sequence $\{\tau_n\}$ for condition (iv).
\begin{proposition}\label{uneqv}
Suppose that $S$ is a bi-integral operator on $L^2$.
Then there exists a unitary operator $U\colon L^2\to L^2$ such that the
operator $T=USU^{-1}$ is a bi-Carleman operator on $L^2$, whose kernel is
a $K^0$-kernel satisfying condition (iv).
\end{proposition}
By virtue of this result, one can confine one's attention (with no loss of generality) to
to second-kind integral equations \eqref{skequ} in which the kernel
$\boldsymbol{T}$ possesses all the properties (i)-(iv).
These four assumptions on $\boldsymbol{T}$ will remain in force for the rest
of the paper, and the notations given in condition (iv) will be used frequently
without warning.
\section{Resolvent $\boldsymbol{K^0}$-Kernels and Their Approximations}
\subsection{Resolvent Kernels for $\boldsymbol{K^0}$-Kernels}
We start with a definition of the resolvent kernel for a $K^0$-kernel,
which is in a sense an alternative to that mentioned in the introduction.
\begin{definition}\label{sf:def} 
Let $\boldsymbol{T}$ be a $K^0$-kernel, let $\lambda$ be a complex number,
and suppose that a $K^0$-kernel, to be denoted by $\boldsymbol{T}_{\mid\lambda}$,
satisfies, for all $s$ and $t$ in $\mathbb{R}$,
the two simultaneous integral equations
\allowdisplaybreaks
\begin{gather}
\boldsymbol{T}_{\mid\lambda}(s,t)-\lambda \int
\boldsymbol{T}(s,x)\boldsymbol{T}_{\mid\lambda}(x,t)\,dx=\boldsymbol{T}(s,t),
\label{eqex}\\
\boldsymbol{T}_{\mid\lambda}(s,t)-\lambda\int
\boldsymbol{T}_{\mid\lambda}(s,x)\boldsymbol{T}(x,t)\,dx=\boldsymbol{T}(s,t),
\label{equn}
\end{gather}
and the condition that, for any $f$ in $L^2$,
\begin{equation}\label{eqbound}
\int\left|\int \boldsymbol{T}_{\mid\lambda}(s,t)f(t)\,dt\right|^2\,ds<\infty.
\end{equation}
Then the $K^0$-kernel $\boldsymbol{T}_{\mid\lambda}$ will be
called the \textit{resolvent kernel} for $\boldsymbol{T}$ at $\lambda$,
and the functions
$\boldsymbol{t}_{\mid\lambda}$,
$\boldsymbol{t}_{\mid\lambda}^{\boldsymbol{\prime}}$
of $C\left(\mathbb{R},L^2\right)$, defined via $\boldsymbol{T}_{\mid\lambda}$ by
$\boldsymbol{t}_{\mid\lambda}(s)=\overline{\boldsymbol{T}_{\mid\lambda}(s,\cdot)}$,
$\boldsymbol{t}^{\boldsymbol{\prime}}_{\mid\lambda}(t)
=\boldsymbol{T}_{\mid\lambda}(\cdot,t)$,
will be called the \textit{resolvent Carleman
functions} for $\boldsymbol{T}$ at $\lambda\in\mathbb{C}$.
\end{definition}
\begin{theorem}\label{rf:ex}
Let $T\in\mathfrak{R}\left(L^2\right)$ be an integral operator, with a kernel
$\boldsymbol{T}$ that is a $K^0$-kernel, and let $\lambda$ be a complex number.
Then {\rm(a)} if $\lambda$ is a regular value for $T$, then the resolvent
kernel for $\boldsymbol{T}$ exists at $\lambda$,
and is a kernel of the Fredholm resolvent of $T$ at $\lambda$, that is,
$\left(T_{\mid\lambda}f\right)(s)=\int\boldsymbol{T}_{\mid\lambda}(s,t)f(t)\,dt$
for every $f$ in $L^2$ and almost every $s$ in $\mathbb{R}$;
{\rm(b)} if the resolvent kernel for $\boldsymbol{T}$  exists at $\lambda$,
then $\lambda$ is a regular value for $T$.
\end{theorem}
\begin{proof}[\indent Proof]
To prove statement (a), let $\lambda$ be an arbitrary but fixed regular value for $T$
($\lambda\in\Pi(T)$), and define two functions $\boldsymbol{a}$, 
$\boldsymbol{a}^{\boldsymbol{\prime}}\colon\mathbb{R}\to L^2$ by writing
\begin{equation}\label{defas}
\boldsymbol{a}(s)=\left(\bar{\lambda} (T_{\mid\lambda})^*+I\right)(\boldsymbol{t}(s)),
\quad
\boldsymbol{a}^{\boldsymbol{\prime}}(s)=
\left(\lambda T_{\mid\lambda}+I\right)\left(\boldsymbol{t}^
{\boldsymbol{\prime}}(s)\right)
\end{equation}
whenever $s\in\mathbb{R}$.
So defined, $\boldsymbol{a}$ and $\boldsymbol{a}^{\boldsymbol{\prime}}$
then belong to the space $C\left(\mathbb{R},L^2\right)$, as
$\boldsymbol{t}$ and $\boldsymbol{t}^{\boldsymbol{\prime}}$
(the Carleman functions associated to the $K^0$-kernel $\boldsymbol{T}$)
are in $C\left(\mathbb{R},L^2\right)$, and $T_{\mid\lambda}$
(the Fredholm resolvent of $T$ at $\lambda$) is in 
$\mathfrak{R}\left(L^2\right)$.
\par
The functions $\boldsymbol{A}$, 
$\boldsymbol{A}^{\boldsymbol{\prime}}\colon\mathbb{R}^2\to\mathbb{C}$,
given by the formulae
\begin{equation}
\begin{gathered}\label{AAprime}
\boldsymbol{A}(s,t)=\lambda\left\langle\boldsymbol{t}^
{\boldsymbol{\prime}}(t),\boldsymbol{a}(s)\right\rangle+\boldsymbol{T}(s,t),
\\
\boldsymbol{A}^{\boldsymbol{\prime}}(s,t)=\bar{\lambda}\overline{\left\langle
\boldsymbol{a}^{\boldsymbol{\prime}}(s),
\boldsymbol{t}(t)\right\rangle}+\overline{\boldsymbol{T}(t,s)},
\end{gathered}
\end{equation}
then belong to the space $C\left(\mathbb{R}^2,\mathbb{C}\right)$,
due to the continuity of the inner product as a function from $L^2\times L^2$ to
$\mathbb{C}$. By using \eqref{defas} it is also seen from  \eqref{AAprime}
that these functions are conjugate transposes of each other, viz. $\boldsymbol{A}^{\boldsymbol{\prime}}(s,t)=
\overline{\boldsymbol{A}(t,s)}$ for all $s$, $t\in\mathbb{R}$.
Simple manipulations involving formulae \eqref{AAprime}, \eqref{TfT*f},
and \eqref{defas} give rise to the following two strings of equations
being satisfied
at all points $s$ in $\mathbb{R}$ by any function $f$ in $L^2$:
\begin{align*}
\begin{split}
\int\boldsymbol{A}(s,t)f(t)\,dt&=\lambda\int\left\langle\boldsymbol{t}^
{\boldsymbol{\prime}}(t),\boldsymbol{a}(s)\right\rangle f(t)\,dt+
\int\boldsymbol{T}(s,t)f(t)\,dt
\\
&=\left\langle f,\overline\lambda T^*(\boldsymbol{a}(s))+\boldsymbol{t}(s)
\right\rangle=\langle f,\boldsymbol{a}(s)\rangle,
\end{split}\label{Ah}
\\
\begin{split}
\int\boldsymbol{A}^{\boldsymbol{\prime}}(s,t)f(t)\,dt&=
\overline{\lambda}\int\overline{\left\langle
\boldsymbol{a}^{\boldsymbol{\prime}}(s),
\boldsymbol{t}(t)\right\rangle} f(t)\,dt+\int\overline{\boldsymbol{T}(t,s)}f(t)\,dt
\\&=
\left\langle f,\lambda T\left(\boldsymbol{a}^{\boldsymbol{\prime}}(s)\right)+
\boldsymbol{t}^{\boldsymbol{\prime}}(s)\right\rangle=
\left\langle f,\boldsymbol{a}^{\boldsymbol{\prime}}(s)\right\rangle.
\end{split}
\end{align*}
The equality of the extremes of each of these strings
implies that $\overline{\boldsymbol{A}(s,\cdot)}\in\boldsymbol{a}(s)$,
$\boldsymbol{A}(\cdot,s)\in\boldsymbol{a}^{\boldsymbol{\prime}}(s)$
for every fixed $s$ in $\mathbb{R}$.
Furthermore, the following relations hold whenever $f$ is in $L^2$:  
\begin{equation}\label{eqboundv}
\begin{split}
\int\boldsymbol{A}(\cdot,t)f(t)\,dt&=\langle f,\boldsymbol{a}(\cdot)\rangle
=
\left\langle\left(\lambda T_{\mid\lambda}+I\right)f,\boldsymbol{t}(\cdot)\right\rangle
\\&=\left\langle R_\lambda(T) f,\boldsymbol{t}(\cdot)\right\rangle
=\left(TR_\lambda(T)f\right)(\cdot)=\left(T_{\mid\lambda} f\right)(\cdot)\in L^2,
\end{split}
\end{equation}
showing that the Fredholm resolvent $T_{\mid\lambda}$ of $T$ at $\lambda$
is an integral operator on $L^2$, with the function $\boldsymbol{A}$ as its kernel
(compare this with \eqref{eqbound}).
\par
The inner product when written in the integral form and the above observations
about $\boldsymbol{A}$ allow the defining relationships  for
$\boldsymbol{A}$ and $\boldsymbol{A}^{\boldsymbol{\prime}}$ (see \eqref{AAprime}) 
to be respectively written as the integral equations
\begin{equation*}
\begin{gathered}
\boldsymbol{A}(s,t)=\lambda\int\boldsymbol{A}(s,x)\boldsymbol{T}(x,t)\,dx+\boldsymbol{T}(s,t),
\\
\boldsymbol{A}(s,t)=\lambda\int\boldsymbol{T}(s,x)\boldsymbol{A}(x,t)\,dx+\boldsymbol{T}(s,t),
\end{gathered}
\end{equation*}
holding for all $s$, $t\in\mathbb{R}$. Together with \eqref{eqboundv}, these
imply that the $K^0$-kernel $\boldsymbol{A}$ is a resolvent kernel for
$\boldsymbol{T}$ at $\lambda$ (in the sense of Definition~\ref{sf:def}).
\par
To prove statement {\rm(b)}, let there exist a $K^0$-kernel
$\boldsymbol{T}_{\mid\lambda}$ satisfying \eqref{eqex} through \eqref{eqbound}.
It is to be proved that $\lambda$ belongs to $\Pi(T)$, that is, that the operator
$I-\lambda T$ is invertible. To this effect, therefore, remark first that
the integral operator $A$ given by $(Af)(s)=\int\boldsymbol{T}_{\mid\lambda}(s,t)f(t)\,dt$
is bounded from $L^2$ into $L^2$, owing to condition \eqref{eqbound} and to
Banach's Theorem (see \cite[p.\ 14]{Halmos:Sun}). Then, due to the multipliability 
property of $K^0$-kernels (see \eqref{fkerPQ}), the kernel equations
\eqref{eqex} and \eqref{equn} give rise to the operator equalities
$(I-\lambda T)A=T$ and $A(I-\lambda T)=T$, respectively. The latter are easily
seen to be equivalent respectively to the following ones
$(I-\lambda T)(I+\lambda A)=I$ and $(I+\lambda A)(I-\lambda T)=I$,
which together imply that the operator $I-\lambda T$ is invertible with
inverse $I+\lambda A$. The theorem is proved.
\end{proof}
\begin{remark} The proof just given establishes that resolvent kernels in the 
sense of Definition~\ref{sf:def} are in one-to-one correspondence with Fredholm 
resolvents. In view of this correspondence: (1)
$\Pi(T)$ might be defined as the set of all those $\lambda\in\mathbb{C}$ at
which the resolvent kernel in the sense of Definition~\ref{sf:def} exists
(thus, whenever $\boldsymbol{T}_{\mid\lambda}$, $\boldsymbol{t}_{\mid\lambda}$, or
$\boldsymbol{t}^{\boldsymbol{\prime}}_{\mid\lambda}$ appear in what follows,
it may and will always be understood that $\lambda$ belongs to $\Pi(T)$);
(2) the resolvent kernel
$\boldsymbol{T}_{\mid\lambda}$ for the $K^0$-kernel $\boldsymbol{T}$ at
$\lambda$ might as well be defined as that $K^0$-kernel which induces
$T_{\mid\lambda}$, the Fredholm resolvent at $\lambda$ of that integral operator 
$T$ whose kernel is $\boldsymbol{T}$. Using \eqref{eqresf} and \eqref{frimlt}, 
the values of the resolvent Carleman functions for $\boldsymbol{T}$ at each 
fixed regular value $\lambda\in\Pi(T)$ can therefore be ascertained by writing
\begin{equation}\label{carf:resk}
\boldsymbol{t}_{\mid\lambda}(\cdot)=R^*_\lambda(T)(\boldsymbol{t}(\cdot)),\quad
\boldsymbol{t}^{\boldsymbol{\prime}}_{\mid\lambda}(\cdot)=
R_\lambda(T)\left(\boldsymbol{t}^{\boldsymbol{\prime}}(\cdot)\right),
\end{equation}
where $\boldsymbol{t}$ and $\boldsymbol{t}^{\boldsymbol{\prime}}$ are Carleman
functions corresponding to $\boldsymbol{T}$ (compare with \eqref{defas} via 
\eqref{resFres}). The resolvent kernel $\boldsymbol{T}_{\mid\lambda}$ for 
$\boldsymbol{T}$, in its turn, can be exactly recovered from the knowledge of 
the resolvent Carleman functions $\boldsymbol{t}_{\mid\lambda}$ and 
$\boldsymbol{t}^{\boldsymbol{\prime}}_{\mid\lambda}$ by the formulae
\begin{equation}\label{ke1.3.6}
\begin{gathered}
\overline{\boldsymbol{T}_{\mid\lambda}(s,t)}=\bar\lambda
\left\langle\boldsymbol{t}_{\mid\lambda}(s),
\boldsymbol{t}^{\boldsymbol{\prime}}(t)\right\rangle
+\overline{\boldsymbol{T}(s,t)}, 
\\
\boldsymbol{T}_{\mid\lambda}(s,t)=\lambda
\left\langle\boldsymbol{t}^{\boldsymbol{\prime}}_{\mid\lambda}(t),
\boldsymbol{t}(s)\right\rangle+\boldsymbol{T}(s,t),
\end{gathered}
\end{equation}
respectively (compare with \eqref{AAprime}). Formulae
\eqref{carf:resk}-\eqref{ke1.3.6} will be useful in what follows.
\end{remark}
\subsection{Resolvent Kernels for Sub-$\boldsymbol{K^0}$-Kernels}
Here, as subsequently, we shall denote the resolvent kernels at $\lambda$ for the
subkernel $\boldsymbol{T}_n$ (resp., $\boldsymbol{\widetilde{T}}_n$)
by $\boldsymbol{T}_{n\mid\lambda}$ (resp.,
$\boldsymbol{\widetilde{T}}_{n\mid\lambda}$), and the  resolvent
Carleman functions for these subkernels at $\lambda$ by
$\boldsymbol{t}_{n\mid\lambda}$,
$\boldsymbol{t}^{\boldsymbol{\prime}}_{n\mid\lambda}$ (resp.,
$\boldsymbol{\widetilde{t}}_{n\mid\lambda}$,
$\boldsymbol{\widetilde{t}}^{\boldsymbol{\prime}}_{n\mid\lambda}$).
Then the following formulae are none other than
valid versions of \eqref{carf:resk}  
and \eqref{ke1.3.6} for $\boldsymbol{t}_{n\mid\lambda_n}$,
$\boldsymbol{t}^{\boldsymbol{\prime}}_{n\mid\lambda_n}$,
and $\boldsymbol{T}_{n\mid\lambda_n}$,
developed making use of \eqref{kesubcf}:
\begin{gather}
\boldsymbol{t}\sb{n\mid\lambda_n}(s)
=\overline{\boldsymbol{T}\sb{n\mid\lambda_n}(s,\cdot)}=
R^*_{\lambda_n}\left({T}_n\right)
\left(\boldsymbol{t}_n(s)\right)=
\chi_n(s)R^*_{\lambda_n}\left({T}_n\right)
\left(\boldsymbol{t}(s)\right),
\label{ke1.3.4sbk}
\\
\boldsymbol{t}^{\boldsymbol{\prime}}_{n\mid\lambda_n}(t)
=\boldsymbol{T}_{n\mid\lambda_n}(\cdot,t)
=R_{\lambda_n}\left({T}_n\right)
\left(\boldsymbol{t}_n^{\boldsymbol{\prime}}(t)\right)
=R_{\lambda_n}\left({T}_n\right)
P_n\left(\boldsymbol{t}^{\boldsymbol{\prime}}(t)\right),
\label{ke1.3.4sbkpr}
\\%%%%%%%%%%%%%%%%%%%%%%%%%%%%%%%%%%%%%%%%%%%%%%%%%
\overline{\boldsymbol{T}_{n\mid\lambda_n}(s,t)}=\bar\lambda_n
\left\langle \boldsymbol{t}_{n\mid\lambda_n}(s),P_n
\left(\boldsymbol{t}^{\boldsymbol{\prime}}(t)\right)\right\rangle
+\overline{\boldsymbol{T}_n(s,t)},
\label{ke1.3.5sbk}\notag
\\
\boldsymbol{T}_{n\mid\lambda_n}(s,t)=\lambda_n\chi_n(s)
\left\langle\boldsymbol{t}^{\boldsymbol{\prime}}_{n\mid\lambda_n}(t),
\boldsymbol{t}(s)\right\rangle+\boldsymbol{T}_n(s,t),
\label{ke1.3.6sbk}
\end{gather}
for all $s$, $t\in\mathbb{R}$. It is readily seen from \eqref{ke1.3.3} and \eqref{ke1.3.6sbk}  that each $K^0$-kernel
$\boldsymbol{T}_{n\mid\lambda}$ has compact $s$-support (namely, lying in
$[-\tau_n,\tau_n]$),
so the condition \eqref{eqbound} of Definition~\ref{sf:def}
is automatically satisfied with $\boldsymbol{T}_{n\mid\lambda}$
in the role of $\boldsymbol{T}_{\mid\lambda}$.
Thus, $\boldsymbol{T}_{n\mid\lambda}$
is the only solution of the simultaneous integral equations \eqref{eqex} and \eqref{equn}
(with $\boldsymbol{T}$ replaced by $\boldsymbol{T}_n$) which is a $K^0$--kernel.
The problem of explicitly finding  that solution in terms of
$\boldsymbol{T}_n$ is completely solved via the Fredholm-determinant method,
as follows.
For $\boldsymbol{T}_n$ a subkernel of $\boldsymbol{T}$,
consider its Fredholm determinant $D_{\boldsymbol{T}_n}(\lambda)$ defined
by the series
\begin{equation}\label{Frdet}
D_{\boldsymbol{T}_n}(\lambda):=1+
             \sum_{m=1}^\infty\frac{(-\lambda)^m}{m!}
\int\dots\int
\boldsymbol{T}_n\begin{pmatrix}
x_{1}&\hdots&x_{m}\\
x_{1}&\hdots&x_{m}\end{pmatrix} dx_1\dots\,dx_m,
\end{equation}
for every $\lambda\in\mathbb{C}$, and its first Fredholm minor
$D_{\boldsymbol{T}_n}(s,t\mid\lambda)$ defined by the series
\begin{equation}\label{Frmin}
D_{\boldsymbol{T}_n}(s,t\mid\lambda)=\boldsymbol{T}_n(s,t)
+\sum_{m=1}^\infty\frac{(-\lambda)^m}{m!}
\int\dots\int\boldsymbol{T}_n\begin{pmatrix}
                      s&x_1&\hdots&x_m\\
                      t&x_1&\hdots&x_m\end{pmatrix}\,
dx_1\dots\,dx_m,
\end{equation}
for all points $s$, $t\in\mathbb{R}$ and for every $\lambda\in\mathbb{C}$,
where
\begin{equation*}
\boldsymbol{T}_n\begin{pmatrix}
x_1&\hdots&x_\nu\\
y_1&\hdots&y_\nu
\end{pmatrix}:=
                                    \det\begin{pmatrix}
                            \boldsymbol{T}_n(x_1,y_1)&\hdots&\boldsymbol{T}_n(x_1,y_\nu)\\
                            \hdotsfor{3}\\
                            \boldsymbol{T}_n(x_\nu,y_1)&\hdots&\boldsymbol{T}_n(x_\nu,y_\nu)\end{pmatrix}.
\end{equation*}
The next proposition can be inferred from results of the
Carleman-Mikhlin-Smithies theory  of the Fredholm determinant and the first Fredholm minor
for Hilbert-Schmidt kernels of possibly unbounded support
(see \cite{Carl:pap}, \cite{Mih:pap}, and \cite{Smith:paper}).
\begin{proposition}\label{FredTh} Let $\lambda\in\mathbb{C}$ be arbitrary but fixed.
Then
\par
1) the series of \eqref{Frdet} is absolutely convergent in $\mathbb{C}$, and
the series of \eqref{Frmin} is absolutely convergent in
$C\left(\mathbb{R}^{2},\mathbb{C}\right)$
and in $L^2\left(\mathbb{R}^2\right)$;
\par
2) if $D_{\boldsymbol{T}_n}(\lambda)\ne 0$ then
resolvent kernel for $\boldsymbol{T}_n$ at $\lambda$ exists and is
the quotient of the first Fredholm minor and the Fredholm determinant:
\begin{equation}\label{ke1.3.7}
\boldsymbol{T}_{n\mid\lambda}(s,t)\equiv\frac
{D_{\boldsymbol{T}_n}(s,t\mid\lambda)}
{D_{\boldsymbol{T}_n}(\lambda)};
\end{equation}
\par
3) if $D_{\boldsymbol{T}_n}(\lambda)=0$, then the resolvent kernel for
$\boldsymbol{T}_n$ does not exist at $\lambda$.
\end{proposition}
For each $n\in\mathbb{N}$, therefore, the characteristic set $\Lambda({T}_n)$
is composed of all the zeros of the entire function $D_{\boldsymbol{T}_n}(\lambda)$, and
is an at most denumerable set clustering at $\infty$.
The Fredholm representation (like \eqref{ke1.3.7}) for $\boldsymbol{\widetilde{T}}_{n\mid\lambda_n}$
is built up in the same way but replacing $\boldsymbol{T}_n$ by
$\boldsymbol{\widetilde{T}}_n$. Since $\widetilde{T}_n^m={T}_n^mP_n$ for
$m\in\mathbb{N}$, the $m$-th iterants of $\boldsymbol{\widetilde{T}}_n$ and
$\boldsymbol{T}_n$ (see \eqref{iterant}) stand therefore in a similar
relation to each other, namely: 
$\boldsymbol{\widetilde{T}}^{[m]}_n(s,t)=
\chi_n(t)\boldsymbol{T}^{[m]}_n(s,t)$ for all $s$, $t\in\mathbb{R}$.
Then it follows from the rules for calculating the coefficients of powers of $\lambda$ in the
Fredholm series (see \eqref{Frdet}, \eqref{Frmin}) that
$D_{\boldsymbol{\widetilde{T}}_n}(\lambda)
\equiv D_{\boldsymbol{T}_n}(\lambda)$,
$D_{\boldsymbol{\widetilde{T}}_n}(s,t\mid\lambda)
\equiv
\chi_n(t)D_{\boldsymbol{T}_n}(s,t\mid\lambda)$.
Hence, for each $n$,
\begin{gather}
\label{ke1.3.8}
\Lambda(T_n)=\Lambda(\widetilde{T}_n),\quad\Pi(T_n)=\Pi(\widetilde{T}_n),
\\
\label{ke1.3.9}
\boldsymbol{\widetilde{T}}_{n\mid\lambda_n}(s,t)
=\frac
{D_{\boldsymbol{\widetilde{T}}_n}(s,t\mid\lambda_n)}
{D_{\boldsymbol{\widetilde{T}}_n}(\lambda_n)}
=\chi_n(t)\boldsymbol{\boldsymbol{T}}_{n\mid\lambda_n}(s,t)
\quad (\lambda_n\in\Pi(T_n)),
\\
\label{ke1.3.10}
\boldsymbol{\widetilde{t}}_{n\mid\lambda_n}(s)
=P_n\left(\boldsymbol{t}_{n\mid\lambda_n}(s)\right),\quad
\boldsymbol{\widetilde{t}}^{\boldsymbol{\prime}}_{n\mid\lambda_n}(t)
=\chi_n(t)\boldsymbol{t}^{\boldsymbol{\prime}}_{n\mid\lambda_n}(t).
\end{gather}

\subsection{The Main Result}
Given an arbitrary sequence $\{\lambda_n\}_{n=1}^\infty$ of complex numbers
satisfying $\lambda_n\in\Pi(T_n)$ for each $n$ and converging to some
$\lambda\in\mathbb{C}$,
the $C\left(\mathbb{R}^2,\mathbb{C}\right)$-valued sequence of the resolvent kernels
\begin{equation}
\label{eq2.3}
\left\{\boldsymbol{T}_{n\mid\lambda_n}\right\}_{n=1}^\infty,
\end{equation}
all of whose terms are known explicitly in terms of the original $K^0$-kernel
$\boldsymbol{T}$ via the Fredholm formulae
\eqref{Frdet}-\eqref{ke1.3.7}, and the $C\left(\mathbb{R},L^2\right)$-valued
sequences of the respective Carleman functions
\begin{equation}
\label{eq2.1I}
\left\{\boldsymbol{t}_{n\mid\lambda_n}\right\}_{n=1}^\infty,\quad
\{\boldsymbol{t}_{n\mid\lambda_n}^{\boldsymbol{\prime}}\}_{n=1}^\infty
\end{equation}
are not known to converge in general. If they do converge, relevant questions would be, e.g.:
if the sequence \eqref{eq2.3} converges in $C\left(\mathbb{R}^2,\mathbb{C}\right)$,
possibly up to the extraction of a subsequence, to a function $\boldsymbol{A}$
say, whether $\lambda$ is necessarily a regular value for $T$,
and if $\lambda$ turns out to belong to $\Pi(T)$, whether
$\boldsymbol{A}=\boldsymbol{T}_{\mid\lambda}$. Similar questions can be asked
concerning the sequences of \eqref{eq2.1I}, but we postpone them all to
a later paper. The (in a sense converse) question we deal with in this paper
is: given that the above $\lambda$ is a (nonzero) regular value for $T$, what
further connections between $\{\lambda_n\}$ and $\lambda$ guarantee the
existence, in suitable senses, of the limit-relations
\begin{equation*} 
\boldsymbol{t}_{\mid\lambda}=\lim_{n\to\infty}\boldsymbol{t}_{n\mid\lambda_n},\quad
\boldsymbol{t}_{\mid\lambda}^{\boldsymbol{\prime}}=\lim_{n\to\infty}\boldsymbol{t}_{n\mid\lambda_n}^{\boldsymbol{\prime}},\quad
\boldsymbol{T}_{\mid\lambda}=\lim_{n\to\infty}\boldsymbol{T}_{n\mid\lambda_n}.
\end{equation*}
In the theorem which follows, we characterize one such connection
by means of sets such as $\nabla_\mathfrak{s}(\cdot)$, defined at the end of Subsection~\ref{freres}.
\begin{theorem}\label{thmeq3.1}
Let $\left\{\beta_n\right\}_{n=1}^\infty$ be an arbitrary sequence of complex
numbers satisfying
\begin{equation}\label{eq3.1}
\lim_{n\to\infty}\beta_n=0,
\end{equation}
and define $\lambda_n(\lambda):=\lambda(1-\beta_n\lambda)^{-1}$,
so that one can consider that $\lambda_n(\lambda)\to\lambda$ when $n\to\infty$ for each fixed
$\lambda\in\mathbb{C}$. Then
$\varnothing\not=\nabla_\mathfrak{s}(\{\beta_n I+T_n\})\subseteq
\nabla_\mathfrak{s}(\{\beta_n I+\widetilde{T}_n\})\subseteq\Pi(T)$
and the following limits hold:
\begin{gather}
\boldsymbol{t}^{\boldsymbol{\prime}}_{\mid\lambda}(t)=
\lim_{n\to\infty}
\boldsymbol{{t}}_{n\mid\lambda_n(\lambda)}^{\boldsymbol{\prime}}(t)
\quad(\lambda\in\nabla_\mathfrak{s}(\{\beta_n I+\widetilde{T}_n\}),\ t\in\mathbb{R}),
\label{eq3.3}
\\
\boldsymbol{t}_{\mid\lambda}(s)=
\lim_{n\to\infty}\boldsymbol{{t}}_{n\mid\lambda_n(\lambda)}(s)
\quad(\lambda\in\nabla_\mathfrak{s}(\{\beta_n I+T_n\}),\ s\in\mathbb{R}),
\label{eq3.5}
\\
\boldsymbol{T}_{\mid\lambda}(s,t)=
\lim_{n\to\infty}\boldsymbol{{T}}_{n\mid\lambda_n(\lambda)}(s,t)
\quad(\lambda\in\nabla_\mathfrak{s}(\{\beta_n I+\widetilde{T}_n\}),\ (s,t)\in\mathbb{R}^2),
\label{eq3.4}
\end{gather}
where:
\par
{\rm(a)} the convergence in \eqref{eq3.3} is in the
$C\left(\mathbb{R},L^2\right)$ norm for each fixed
$\lambda\in\nabla_\mathfrak{s}(\{\beta_n I+\widetilde{T}_n\})$
(see \eqref{eq3.3d}), and is uniform in $\lambda$ on every compact subset
$\widetilde{\mathfrak{K}}$ of $\nabla_\mathfrak{s}(\{\beta_n I+\widetilde{T}_n\})$ for each fixed $t\in\mathbb{R}$
(see \eqref{eq3.6d});
\par
{\rm(b)} the convergence in \eqref{eq3.5} is in the $C\left(\mathbb{R},L^2\right)$
norm for each fixed $\lambda\in\nabla_\mathfrak{s}(\{\beta_n I+T_n\})$
(see \eqref{eq3.5d}), and is uniform in $\lambda$ on every compact subset $\mathfrak{K}$ of
$\nabla_\mathfrak{s}(\{\beta_n I+T_n\})$ for each fixed $s\in\mathbb{R}$
(see \eqref{eq3.8d}); and
\par
{\rm(c)} the convergence in \eqref{eq3.4} is in the
$C\left(\mathbb{R}^2,\mathbb{C}\right)$ norm for each fixed
$\lambda\in\nabla_\mathfrak{s}(\{\beta_n I+\widetilde{T}_n\})$
(see \eqref{eq3.4d}), and is
uniform in $\lambda$ on every compact subset $\widetilde{\mathfrak{K}}$ of
$\nabla_\mathfrak{s}(\{\beta_n I+\widetilde{T}_n\})$ for each fixed $(s,t)\in\mathbb{R}^2$
(see \eqref{eq3.7d}).
\end{theorem}
%%%%%%%%%%%%%%%%%%%%%%%%%%%%%%%%%%%%%%%%%%%%%%%%%%%%%
\begin{proof}[\indent Proof]
Let us begin by collecting (mainly from \cite{Kato:book})
some preparatory results, to be numbered below from \eqref{eq3.4new} to
\eqref{eq3.21}.
To simplify the notation, write
${A}_n:=\beta_n I+{T}_n$, $\widetilde{A}_n:=\beta_n I+\widetilde{T}_n$.
Choose a (non-zero) regular value $\zeta\in\Pi(T)$ satisfying
$ \left|\zeta\right|\left\|T\right\|<1$, and hence satisfying
for some $N(\zeta)>0$ the inequality
\begin{equation}\label{eqd3.5}
\left|\zeta\right|\left\|{A}_n\right\|\leqslant \left|\zeta\right|\left(\max_{n>N(\zeta)}\left|\beta_n
                        \right|+\left\|T\right\|\right)<1
\quad\text{for all $n>N(\zeta)$}.
\end{equation}
Then $\zeta$ does belong to $\nabla_\mathfrak{b}(\{A_n\})$,
because 
\begin{equation*}
\left\|{A}_{n\mid\zeta}\right\|\leqslant
                            \frac{\left\|{A}_n\right\|}
                                 {1-\left|\zeta\right|\left\|{A}_n\right\|}
\leqslant M(\zeta)=\frac{\max\limits_{n>N(\zeta)}\left\|{A}_n\right\|}{1-\left|\zeta\right|
\left(\max\limits_{n>N(\zeta)}
\left|\beta_n\right|
                        +\left\|T\right\|\right)}
\quad\text{for all $n>N(\zeta)$} 
\end{equation*}
(cf.~\eqref{eq3.2}).
The result is that the intersection of $\nabla_\mathfrak{b}(\{A_n\})$
and $\Pi(T)$ is non-void. Similarly it can be shown that 
$\nabla_\mathfrak{b}(\{\widetilde{A}_n\})\cap\Pi(T)\not=\varnothing$.
Therefore, since, because of \eqref{eq3.1} and \eqref{eqTntoT}, the sequences
$\{{A}_n\}$ and $\{\widetilde{A}_n\}$ both converge to $T$ in the strong operator topology,
it follows by the criterion for generalized strong convergence 
(see \cite[Theorem~VIII-1.5]{Kato:book}) that
\begin{gather}
\nabla_\mathfrak{s}(\{A_n\})=\nabla_\mathfrak{b}(\{A_n\})\cap\Pi(T), \quad
\nabla_\mathfrak{s}(\{\widetilde{A}_n\})=\nabla_\mathfrak{b}(\{\widetilde{A}_n\})\cap\Pi(T),\label{eq3.4new}\\
\lim_{n\to\infty}\left\|\left({A}_{n\mid\lambda}-T_{\mid\lambda}\right)f\right\|=0
\quad\text{for all $\lambda\in\nabla_\mathfrak{s}(\{A_n\})$ and $f\in L^2$,}\label{eqd3.6} \notag\\
\lim_{n\to\infty}\left\|\left(\widetilde{A}_{n\mid\lambda}-T_{\mid\lambda}\right)f\right\|=0
\quad\text{for all $\lambda\in\nabla_\mathfrak{s}(\{\widetilde{A}_n\})$ and
$f\in L^2$.}\label{eqd3.8}
\end{gather}
\par
Further, given a $\lambda\in\nabla_\mathfrak{b}(\{A_n\})\cup\nabla_\mathfrak{b}(\{\widetilde{A}_n\})$,
the following formulae hold for sufficiently large $n$:
\begin{equation}\label{eq3.26}
\begin{gathered}
R_\lambda({A}_n)=
\frac{1}{1-\beta_n\lambda}R_{\lambda_n(\lambda)}({T}_n), \quad
R_\lambda(\widetilde{A}_n)=\frac{1}{1-\beta_n\lambda}
R_{\lambda_n(\lambda)}(\widetilde{T}_n),
\\
R_{\lambda_n(\lambda)}({T}_n)=
(1-\beta_n\lambda)\left(I+\lambda{A}_{n\mid\lambda}\right)
=
I+\lambda\beta_n I +\lambda{A}_{n\mid\lambda}+
\lambda^2\beta_n{A}_{n\mid\lambda},
\\
{A}_{n\mid\lambda}
=\left(\frac{1}{1-\beta_n\lambda}\right)^2{T}_{n\mid\lambda_n(\lambda)}
+\frac{\beta_n}{1-\beta_n\lambda} I,
\\
\widetilde{A}_{n\mid\lambda}=\left(\frac{1}{1-\beta_n\lambda}\right)^2\widetilde{T}_{n\mid\lambda_n(\lambda)}
+\frac{\beta_n}{1-\beta_n\lambda} I.
\end{gathered}
\end{equation}
These are obtained by a purely formal calculation, and use that fact that
$\Pi(\widetilde{T}_n)=\Pi({T}_n)$ for each fixed $n\in\mathbb{N}$
(see \eqref{ke1.3.8}).
%%%%%%%%%%%%%%%%%%%%%%%%%%%%%%%%%%%%%%%%%%%%%%%%%%%%%%%%%%%%%%%%%%%%%%%%%
The equations in the last two lines combine to give, using \eqref{ke1.3.9},
\begin{equation}\label{eqd3.10}
{A}_{n\mid\lambda}P_n=\widetilde{A}_{n\mid\lambda}
+\frac{\beta_n}{1-\beta_n\lambda}(I-P_n).
\end{equation}
This implies in particular that
$\left\|\widetilde{A}_{n\mid\lambda}\right\|\leqslant\left\|{A}_{n\mid\lambda}\right\|
+|\beta_n||1-\beta_n\lambda|^{-1}$,
whence \eqref{eq3.1} leads to the inclusion relation
$\nabla_\mathfrak{b}(\{A_n\})\subseteq\nabla_\mathfrak{b}(\{\widetilde{A}_n\})$,
from which it follows via \eqref{eq3.4new} that
$\varnothing\not=\nabla_\mathfrak{s}(\{A_n\})\subseteq\nabla_\mathfrak{s}(\{\widetilde{A}_n\})\subseteq\Pi(T)$,
as asserted.
\par
In what follows, let $\widetilde{\mathfrak{K}}$  denote a compact subset of
$\nabla_\mathfrak{s}(\{\widetilde{A}_n\})$. Then, according to  Theorem~VIII-1.1 
in \cite{Kato:book} there exists a positive constant $M(\widetilde{\mathfrak{K}})$ 
such that
\begin{equation}\label{eqd3.16}
\sup_{\lambda\in\widetilde{\mathfrak{K}}}\left\|\widetilde{A}_{n\mid\lambda}\right\|\leqslant M(\widetilde{\mathfrak{K}})
\quad\text{for all sufficiently large $n$},
\end{equation}
and, according to Theorem~VIII-1.2 therein,
the convergence in \eqref{eqd3.8} is uniform over $\widetilde{\mathfrak{K}}$:
\begin{equation}\label{eq3.17}
\lim_{n\to\infty}\sup_{\lambda\in\widetilde{\mathfrak{K}} }
\left\|\left(\widetilde{A}_{n\mid\lambda}-T_{\mid\lambda}\right)f\right\|=0\quad
\text{for each fixed $f\in L^2$.}
\end{equation}
Now use \eqref{eqd3.16}, \eqref{eq3.17}, and the observation from
\eqref{eq3.1} that
\begin{equation}\label{eqesbet}
\sup\limits_{\lambda\in\widetilde{\mathfrak{K}}}\left|\frac{\beta_n}{1-\beta_n\lambda}\right|
\leqslant \frac{\left|\beta_n\right|}
{1-\left|\beta_n\right|\sup\limits_{\lambda\in\widetilde{\mathfrak{K}}}\left|\lambda\right|}\to0
\quad\text{as $n\to\infty$},
\end{equation}
to infer, via the connecting formula \eqref{eqd3.10}, that
\begin{gather}
\lim_{n\to\infty}\sup_{\lambda\in\widetilde{\mathfrak{K}}}
\left\|\left({A}_{n\mid\lambda}P_n-T_{\mid\lambda}\right)f\right\|=0\quad
\text{for each fixed $f\in L^2$,}\label{eq3.19}
\\
\sup_{\lambda\in\widetilde{\mathfrak{K}}}\left\|{A}_{n\mid\lambda}P_n\right\|<M(\widetilde{\mathfrak{K}})+1
\quad\text{for all sufficiently large $n$.}
\label{eqadsup}
\end{gather}
\par
Throughout what follows let $\mathfrak{K}$ denote a compact subset of
$\nabla_\mathfrak{s}(\{A_n\})$. Then Theorem~VIII-1.1 in \cite{Kato:book}, this time
applied to the operator sequence $\{{A}_n\}$, yields the conclusion that
there exists a positive constant $M(\mathfrak{K})$ such that
\begin{equation}\label{eq3.22}
\sup_{\lambda\in\mathfrak{K}}\left\|{A}_{n\mid\lambda}\right\|
\leqslant M\left(\mathfrak{K}\right)\quad\text{for all sufficiently large $n$,}
\end{equation}
and hence there holds 
\begin{equation}\label{eq3.21}
\lim_{n\to\infty}\sup_{\lambda\in\mathfrak{K}}
\left\|\left({A}_{n\mid\lambda}
-T_{\mid\lambda}\right)^*f\right\|=0\quad\text{for each fixed $f\in L^2$.}
\end{equation}
Indeed, given any $f\in L^2$, the following relations hold:
\begin{align*} 
&\quad \lim_{n\to\infty}\sup_{\lambda\in{\mathfrak{K}} }
\left\|\left({A}_{n\mid\lambda}-T_{\mid\lambda}\right)^*f\right\|
\notag
\\
&=\lim_{n\to\infty}\sup_{\lambda\in{\mathfrak{K}} }
\left\|\left(I+ \bar\lambda\left({A}_{n\mid\lambda}\right)^*\right)
\left(T-{A}_n\right)^*R_\lambda^*(T)f\right\| && \text{by \eqref{secFResEq}}
\\
&\leqslant
\sup_{\lambda\in{\mathfrak{K}}}(1+|\lambda|M({\mathfrak{K}}))
\lim_{n\to\infty}\sup_{\lambda\in{\mathfrak{K}}}
\left\|\left(T-{A}_n\right)^*R_\lambda^*(T)f\right\| && \text{by \eqref{eqd3.16}}
\notag
\\&=0                     && \text {by Lemma~\ref{lemKato},}\notag
\end{align*}
inasmuch as $\left({A}_n\right)^*\to T^*$ strongly as $n\to\infty$
(see \eqref{eqTntoT}, \eqref{eq3.1})
and the set $\bigcup\limits_{\lambda\in{\mathfrak{K}}}R_\lambda^*(T)f$
is relatively compact in $L^2$ (being the image under the
continuous map $R_\lambda^*(T)f\colon\Pi(T)\to L^2$
(see Subsection~\ref{freres})
of the compact subset ${\mathfrak{K}}$ of $\Pi(T)$). Similarly, it can be proved that
\begin{equation}\label{eq3.19dd}
\lim_{n\to\infty}\sup_{\lambda\in\widetilde{\mathfrak{K}}}
\left\|P_n\left(A_{n\mid\lambda}P_n-T_{\mid\lambda}\right)^*f\right\|=0\quad
\text{for each fixed $f\in L^2$.}
\end{equation}
\par
With these preparations, we are ready to establish that the limit formulae
\eqref{eq3.3}-\eqref{eq3.5} all hold, each uniformly in two senses, exactly as 
stated in the enunciation of the theorem. For this purpose, use formulae 
\eqref{ke1.3.4sbkpr}, \eqref{ke1.3.4sbk}, \eqref{carf:resk}, \eqref{eq3.26}, and then the triangle
inequality to formally write
\begin{multline}\label{eq3.27}
\sup\left\|\boldsymbol{{t}}_{n\mid\lambda_n(\lambda)}^{\boldsymbol{\prime}}(t)
-\boldsymbol{t}^{\boldsymbol{\prime}}_{\mid\lambda}(t)\right\|
\\
=
\sup\left\|\left(P_n-I+\lambda\beta_nP_n+\lambda({A}_{n\mid\lambda}P_n
-T_{\mid\lambda})+\lambda^2\beta_n{A}_{n\mid\lambda}P_n\right)
\left(\boldsymbol{t}^{\boldsymbol{\prime}}(t)\right)\right\|
\\
\leqslant
\sup\left(\left\|(P_n-I)\left(\boldsymbol{t}^{\boldsymbol{\prime}}(t)\right)\right\|\right)
+
\sup\left(\left|\beta_n\right|\left|\lambda\right|\left\|P_n\left(\boldsymbol{t}^{\boldsymbol{\prime}}(t)\right)\right\|\right)
\\
+
\sup\left(\left|\lambda\right|
\left\|\left({A}_{n\mid\lambda}P_n-T_{\mid\lambda}\right)\left(\boldsymbol{t}^{\boldsymbol{\prime}}(t)\right)\right\|\right)
+
\sup\left(\left|\beta_n\right|\left|\lambda\right|^2\left\|{A}_{n\mid\lambda}P_n\right\|
\boldsymbol{\tau}^{\boldsymbol{\prime}}(t)\right),
\end{multline}
\begin{multline}\label{eqcurb}
\sup\left\|\boldsymbol{t}_{n\mid\lambda_n(\lambda)}(s)-
\boldsymbol{t}_{\mid\lambda}(s)\right\|\\
=\sup\left\|\left(\chi_n(s)\left(I+\lambda\beta_nI+\lambda{A}_{n\mid\lambda}
+\lambda^2\beta_n{A}_{n\mid\lambda}\right)
-I-\lambda T_{\mid\lambda}\right)^*\left(\boldsymbol{t}(s)\right)\right\|
\\
\leqslant
\sup\left(\left|\beta_n\right|\chi_n(s)\left|\lambda\right|\boldsymbol{\tau}(s)\right)+
\sup\left(\widehat{\chi}_n(s)\boldsymbol{\tau}(s)\right)
\\+
\sup\left(\chi_n(s)\left|\lambda\right|\left\|\left(
{A}_{n\mid\lambda}-T_{\mid\lambda}\right)^*\left(\boldsymbol{t}(s)\right)\right\|\right)
+
\sup\left(\widehat{\chi}_n(s)\left|\lambda\right|\left\|T_{\mid\lambda}\right\|\boldsymbol{\tau}(s)\right)
\\+
\sup\left(\left|\beta_n\right|\chi_n(s)\left|\lambda\right|^2
\left\|A_{n\mid\lambda}\right\|\boldsymbol{\tau}(s)\right)
\end{multline}
(``formally'' because we have not specified the domain over which the suprema 
are being taken).
Now use equations
\eqref{ke1.3.6}, \eqref{ke1.3.6sbk}, and the triangle and the Cauchy-Schwarz 
inequality to also formally write
\allowdisplaybreaks
\begin{multline}\label{eqGamb}
\sup\left|\boldsymbol{{T}}_{n\mid\lambda_n(\lambda)}(s,t)-\boldsymbol{T}_{\mid\lambda}(s,t)\right|
\\
=
\sup\left|\lambda_n(\lambda)\chi_n(s)
\left\langle \boldsymbol{{t}}^{\boldsymbol{\prime}}_{n\mid\lambda_n(\lambda)}(t),\boldsymbol{t}(s)\right\rangle
-\lambda\left\langle\boldsymbol{t}^{\boldsymbol{\prime}}_{\mid\lambda}(t),\boldsymbol{t}(s)\right\rangle
+\boldsymbol{{T}}_n(s,t)-\boldsymbol{T}(s,t)\right|
\\
\leqslant
\sup\left(\chi_n(s)\left|\lambda\right|\left|\left\langle
\boldsymbol{{t}}_{n\mid\lambda_n(\lambda)}^{\boldsymbol{\prime}}(t)
-\boldsymbol{t}^{\boldsymbol{\prime}}_{\mid\lambda}(t),
\boldsymbol{t}(s)\right\rangle\right|\right)
\\
+
\sup\left(\chi_n(s)\left|\lambda_n(\lambda)-\lambda\right|\left|
\left\langle \boldsymbol{{t}}_{n\mid\lambda_n(\lambda)}^{\boldsymbol{\prime}}(t),
\boldsymbol{t}(s)\right\rangle\right|\right)
\\+
\sup\left(\widehat{\chi}_n(s)\left|\lambda\right|
\left|\left\langle\boldsymbol{t}^{\boldsymbol{\prime}}_{\mid\lambda}(t),\boldsymbol{t}(s)\right\rangle\right|\right)
+
\sup\left|\boldsymbol{{T}}_n(s,t)-\boldsymbol{T}(s,t)\right|
\\
\leqslant
\sup\left(\chi_n(s)\left|\lambda\right|
\left\|\boldsymbol{{t}}_{n\mid\lambda_n(\lambda)}^{\boldsymbol{\prime}}(t)-
\boldsymbol{t}^{\boldsymbol{\prime}}_{\mid\lambda}(t)\right\|\boldsymbol{\tau}(s)\right)
\\+
\sup\left(\chi_n(s)
\left|\lambda\right|^2\left|\frac{\beta_n}{1-\beta_n\lambda}\right|
\left\|\boldsymbol{{t}}_{n\mid\lambda_n(\lambda)}^{\boldsymbol{\prime}}(t)\right\|\boldsymbol{\tau}(s)\right)
\\+
\sup\left(\widehat{\chi}_n(s)\left|\lambda\right|\left\|R_\lambda(T)\right\|
\boldsymbol{\tau}^{\boldsymbol{\prime}}(t)
\boldsymbol{\tau}(s)\right)
+
\sup\left|\boldsymbol{{T}}_n(s,t)-\boldsymbol{T}(s,t)\right|.
\end{multline}
\par
\textrm{(a)}
For a fixed $\lambda\in\nabla_\mathfrak{s}(\{\widetilde{A}_n\})$
take the suprema in \eqref{eq3.27} over all $t\in\mathbb{R}$.
Then each summand on the right-hand side of \eqref{eq3.27}
becomes an $n$-th term of a null sequence of $C\left(\mathbb{R},L^2\right)$-norm values,
by means of \eqref{eqPntoI}, \eqref{eq3.19}, \eqref{eq3.1}, and \eqref{Katolem}.
This proves \eqref{eq3.3} in the following uniform version:
\begin{equation}\label{eq3.3d}
\lim_{n\to\infty}\left\|\boldsymbol{{t}}_{n\mid\lambda_n(\lambda)}^{\boldsymbol{\prime}}
-\boldsymbol{t}^{\boldsymbol{\prime}}_{\mid\lambda}\right\|_{C\left(\mathbb{R},L^2\right)}=0
\quad\text{for each fixed $\lambda\in\nabla_\mathfrak{s}(\{\beta_n I+\widetilde{T}_n\})$}. 
\end{equation}
\par
Next, because of \eqref{eqPntoI}, \eqref{eq3.19}, \eqref{eqadsup}, \eqref{eq3.1}, and of the
boundedness of the set $\widetilde{\mathfrak{K}}$, the suprema at the right-hand 
side of \eqref{eq3.27}, all taken, this time, over all 
$\lambda\in\widetilde{\mathfrak{K}}$, tend as $n\to\infty$ to zero,
which proves that the limit \eqref{eq3.3} holds in the sense that
\begin{equation}\label{eq3.6d}
\lim_{n\to\infty}\sup_{\lambda\in\widetilde{\mathfrak{K}}}
\left\|\boldsymbol{{t}}_{n\mid\lambda_n(\lambda)}^{\boldsymbol{\prime}}(t)
-\boldsymbol{t}^{\boldsymbol{\prime}}_{\mid\lambda}(t)\right\|=0
  \quad\text{for each fixed
  $t\in\mathbb{R}$.}
\end{equation}
\par
\textrm{(b)}
As for the convergence in \eqref{eq3.5}, its uniformity with respect to $s$,
\begin{equation}\label{eq3.5d}
\lim_{n\to\infty}\left\|\boldsymbol{{t}}_{n\mid\lambda_n(\lambda)}-
\boldsymbol{t}_{\mid\lambda}\right\|_{C\left(\mathbb{R},L^2\right)}=0
\quad \text{for each fixed $\lambda\in\nabla_\mathfrak{s}(\{\beta_n I+T_n\})$},
\end{equation}
may be proved similarly to \eqref{eq3.3d}, first taking the suprema in 
\eqref{eqcurb} to be over $\mathbb{R}$ with respect to $s$ and then taking 
account of \eqref{eq3.21}, \eqref{eqnormkf}, \eqref{eq3.1},
and \eqref{Katolem}.
\par
To see that formula \eqref{eq3.5} also holds in its asserted form
\begin{equation}\label{eq3.8d}
\lim_{n\to\infty}\sup_{\lambda\in\mathfrak{K}}
\left\|\boldsymbol{{t}}_{n\mid\lambda_n(\lambda)}(s)
-\boldsymbol{t}_{\boldsymbol{\mid}\lambda}(s)\right\|=0\quad
\text{for each fixed $s\in\mathbb{R}$},
\end{equation}
extend the suprema in \eqref{eqcurb} over $\mathfrak{K}$ with respect to 
$\lambda$ and then apply \eqref{eq3.21}, \eqref{eq3.22}, \eqref{eqnormkf}, 
and \eqref{eq3.1} to the right-hand-side terms there.
\par
\textrm{(c)}
Two uniform versions claimed in the theorem for the limit \eqref{eq3.4} are
written as
\begin{gather}
\label{eq3.4d}
\lim_{n\to\infty}\left\|\boldsymbol{{T}}_{n\mid\lambda_n(\lambda)}-
\boldsymbol{T}_{\mid\lambda}\right\|_{C\left(\mathbb{R}^2,\mathbb{C}\right)}=0
\quad \text{for each fixed $\lambda\in\nabla_\mathfrak{s}(\{\beta_n I+\widetilde{T}_n\})$,}
\\
\label{eq3.7d}
\lim_{n\to\infty}\sup_{\lambda\in\widetilde{\mathfrak{K}}}
\left|\boldsymbol{{T}}_{n\mid\lambda_n(\lambda)}(s,t)
-\boldsymbol{T}_{\mid\lambda}(s,t)\right|=0\quad\text{for each fixed $(s,t)\in\mathbb{R}^2$}
\end{gather}
and will be proved by directly invoking \eqref{eqGamb}.
If the suprema involved therein are taken over all points $(s,t)\in\mathbb{R}^2$,
then the above-established relations \eqref{ke1.3.1}, \eqref{eq3.3d},
\eqref{eqesbet}, and \eqref{eqnormkf} together imply that all four terms on
the extreme right side of \eqref{eqGamb} converge to $0$ as $n\to\infty$,
which proves \eqref{eq3.4d}.
Similarly, the validity of \eqref{eq3.7d} can be deduced from \eqref{eq3.6d}, 
\eqref{ke1.3.1}, \eqref{eqnormkf}, and \eqref{eqesbet} upon taking the suprema 
in \eqref{eqGamb} (with $s$ and $t$ kept fixed) over all $\lambda\in\mathbb{C}$ 
belonging to the bounded set $\widetilde{\mathfrak{K}}$. The theorem is proved.
\end{proof}
%%%%%%%%%%%%%%%%%%%%%%%%%%%%%%%%%%%%%%%%%%%%%%%%%%%%%%%%%%%%%%%%%%%%%%%%%%%
\begin{remark}\label{3.2}
Because of the observation at the beginning of the above proof the
punctured disk
$D_{\|T\|}=\{\lambda\in\mathbb{C}\mid 0<|\lambda|<\tfrac1{\|T\|}\}$
has the property that
\begin{equation*}
D_{\|T\|}\subset\nabla_\mathfrak{s}(\{\beta_n I+T_n\})
\subset\nabla_\mathfrak{s}(\{\beta_n I+\widetilde{T}_n\})
\end{equation*}
for any choice of a complex null sequence $\{\beta_n\}$.
It therefore follows that if $\lambda\in D_{\|T\|}$, the sequence
$\{\lambda_n(\lambda)\}$  figuring in formulae \eqref{eq3.3d}, \eqref{eq3.5d},
and \eqref{eq3.4d} can be replaced by any sequence approaching $\lambda$,
while retaining the uniform convergences. In particular, one can simply take
each $\lambda_n(\lambda)$ equal to $\lambda$.
Meanwhile there is another, more practical, expression 
for $\boldsymbol{T}_{\mid\lambda}$ at $\lambda\in D_{\|T\|}$, which
may be obtained as follows:
\begin{align*}\label{mi5}
\boldsymbol{T}_{\mid\lambda}(s,t)
&=\boldsymbol{T}(s,t)+\lambda\langle R_\lambda(T)
(\boldsymbol{t}^{\boldsymbol{\prime}}(t)),\boldsymbol{t}(s)\rangle
\notag
&& \text{by \eqref{ke1.3.6} and \eqref{carf:resk}}
\\&=
\boldsymbol{T}(s,t)+\lambda\langle \left(\sum_{n=0}^\infty \lambda^nT^n\right)
(\boldsymbol{t}^{\boldsymbol{\prime}}(t)),\boldsymbol{t}(s)\rangle
&& \text{by \eqref{resseries}}
\\&=
\boldsymbol{T}(s,t)+\sum_{n=0}^\infty\left\langle\lambda^{n+1}T^n
(\boldsymbol{t}^{\boldsymbol{\prime}}(t)),\boldsymbol{t}(s)\right\rangle
&& \text{by \eqref{resseries}}
\\&=
\sum_{n=1}^\infty\lambda^{n-1}\boldsymbol{T}^{[n]}(s,t).
&& \text{by \eqref{iterant}}
\end{align*}
The series in the last line is the Neumann series for $\boldsymbol{T}_{\mid\lambda}$;
it is convergent to $\boldsymbol{T}_{\mid\lambda}$ in
$C\left(\mathbb{R}^2,\mathbb{C}\right)$ for $\lambda$ satisfying
\eqref{resseries}, as
\allowdisplaybreaks
\begin{equation*}
\left\|\boldsymbol{T}^{[n]}\right\|_{C\left(\mathbb{R}^2,\mathbb{C}\right)}=
\sup_{(s,t)\in\mathbb{R}^2}\left|\langle T^{n-2}
\boldsymbol{t}^{\boldsymbol{\prime}}(t),\boldsymbol{t}(s)\rangle\right|
\leqslant
\|\boldsymbol{\tau}^{\boldsymbol{\prime}}\|_{C(\mathbb{R},\mathbb{R})}
\|\boldsymbol{\tau}\|_{C(\mathbb{R},\mathbb{R})}\left\|T^{n-2}\right\|.
\end{equation*}
\end{remark}
\begin{remark}\label{remeq3.3}
Applying the respective results of Theorem~\ref{thmeq3.1} in conjunction with
the inequalities
\begin{gather*}
\left\|\boldsymbol{\widetilde{t}}_{n\mid\lambda_n(\lambda)}^{\boldsymbol{\prime}}(t)
-\boldsymbol{t}^{\boldsymbol{\prime}}_{\mid\lambda}(t)\right\|\leqslant\chi_n(t)
\left\|\boldsymbol{{t}}_{n\mid\lambda_n(\lambda)}^{\boldsymbol{\prime}}(t)
-\boldsymbol{t}^{\boldsymbol{\prime}}_{\mid\lambda}(t)\right\|+
\widehat{\chi}_n(t)\left\|\boldsymbol{t}^{\boldsymbol{\prime}}_{\mid\lambda}(t)\right\|,
\\
\left|\boldsymbol{\widetilde{T}}_{n\mid\lambda_n(\lambda)}(s,t)-\boldsymbol{T}_{\mid\lambda}(s,t)\right|\leqslant
\chi_n(t)\left|\boldsymbol{{T}}_{n\mid\lambda_n(\lambda)}(s,t)-\boldsymbol{T}_{\mid\lambda}(s,t)\right|
+\widehat{\chi}_n(t)\left|\boldsymbol{T}_{\mid\lambda}(s,t)\right|
\end{gather*}
(see \eqref{ke1.3.9}, \eqref{ke1.3.10}) yields that
the limits \eqref{eq3.3d}, \eqref{eq3.6d}, \eqref{eq3.4d}, and \eqref{eq3.7d}
all remain valid upon replacing
$\boldsymbol{{t}}_{n\mid\lambda_n(\lambda)}^{\boldsymbol{\prime}}$
and $\boldsymbol{{T}}_{n\mid\lambda_n(\lambda)}$
by
$\boldsymbol{\widetilde{t}}_{n\mid\lambda_n(\lambda)}^{\boldsymbol{\prime}}$ and
$\boldsymbol{\widetilde{T}}_{n\mid\lambda_n(\lambda)}$, respectively.
In turn, the limits \eqref{eq3.5d} and \eqref{eq3.8d} continue to hold with
$\boldsymbol{{t}}_{n\mid\lambda_n(\lambda)}$ and $\nabla_\mathfrak{s}(\{\beta_n I+T_n\})$
replaced respectively by $\boldsymbol{\widetilde{t}}_{n\mid\lambda_n(\lambda)}$
and $\nabla_\mathfrak{s}(\{\beta_n I+\widetilde{T}_n\})$, and to prove this
use can be made of the inequalities
\begin{multline*}
\left\|\boldsymbol{\widetilde{t}}_{n\mid\lambda_n(\lambda)}(s)
-\boldsymbol{t}_{\mid\lambda}(s)\right\|
\leqslant\left\|P_n(\boldsymbol{t}_{n\mid\lambda_n(\lambda)}(s)
-\boldsymbol{t}_{\mid\lambda}(s))\right\|+
\left\|(I-P_n)\left(\boldsymbol{t}_{\mid\lambda}(s)\right)\right\|,
\\
\left\|P_n(\boldsymbol{t}_{n\mid\lambda_n(\lambda)}(s)-
\boldsymbol{t}_{\mid\lambda}(s))\right\|
\\
=\left\|P_n\left(\chi_n(s)\left(I+\lambda\beta_nI+\lambda{A}_{n\mid\lambda}
+\lambda^2\beta_n{A}_{n\mid\lambda}\right)
-I-\lambda T_{\mid\lambda}\right)^*\left(\boldsymbol{t}(s)\right)\right\|
\\
\leqslant
\left|\beta_n\right|\chi_n(s)\left|\lambda\right|\boldsymbol{\tau}(s)+
\widehat{\chi}_n(s)\boldsymbol{\tau}(s)
\\
+
\chi_n(s)\left|\lambda\right|\left\|P_n\left(
{A}_{n\mid\lambda}-T_{\mid\lambda}\right)^*\left(\boldsymbol{t}(s)\right)\right\|
+
\widehat{\chi}_n(s)\left|\lambda\right|\left\|T_{\mid\lambda}\right\|\boldsymbol{\tau}(s)
\\+
\left|\beta_n\right|\chi_n(s)\left|\lambda\right|^2
\left\|A_{n\mid\lambda}P_n\right\|\boldsymbol{\tau}(s)
\end{multline*}
(cf.~\eqref{eqcurb}) and of the properties \eqref{eq3.19dd} and \eqref{eqadsup}.
\end{remark}
\par
In connection with Theorem~\ref{thmeq3.1} the following natural question can be 
asked: in what cases are the sets $\overset{\circ}\Pi(T):=\Pi(T)\setminus\{0\}$ and
$\nabla_\mathfrak{s}(\{\frac{\lambda_n-\lambda}{\lambda\lambda_n} I+T_n\})$ coincident? 
One answer to this question is given in the following theorem.
\begin{theorem}\label{thmeq3.4}
Suppose that
\begin{equation}\label{eq3.9}
\left\|(T-T_n)T_n^m\right\|\to 0\quad\text{as $n\to\infty$,}
\end{equation}
for some $m$ in $\mathbb{N}$. Then
\begin{equation}\label{eq3.11}
\nabla_{\mathfrak{s}}(\{\beta_nI+T_n\})=\overset{\circ}\Pi(T)
\subset\nabla_{\mathfrak{b}}(\{\beta_nI+T_n\})
\end{equation}
for any choice of a sequence
$\left\{\beta_n\right\}$ converging to $0$.
\end{theorem}
\begin{proof}[\indent Proof]
Continue to denote $A_n:=\beta_n I+T_n$ as in the previous proof.
Let $\lambda$ be a fixed non-zero regular value for $T$.
A straightforward calculation yields the equation
\begin{equation}\label{eq3.12}
\begin{gathered}
\left((I-\lambda T)
\sum\limits_{k=0}^{m-1}\lambda^k{A}_n^k+\lambda^m{A}_n^m\right)
\left(I-\lambda{A}_n\right)
\\\quad
=(I-\lambda T)\left(I+\lambda^{m+1}R_\lambda(T)(T-{A}_n){A}_n^m\right).
\end{gathered}
\end{equation}
Expanding binomially $\left(\beta_n I+T_n\right)^m$ and
utilizing conditions \eqref{eq3.1} and \eqref{eq3.9}
gives
\begin{multline*}
\left\|(T-{A}_n){A}_n^m\right\|=
\left\|(T-\beta_nI-T_n)
\left(\beta_nI+T_n\right)^m\right\|
\\
\leqslant\left\|(T-T_n)T_n^m\right\|+
\left\|\beta_nT_n^m\right\|
\\ +\left\|(T-\beta_n I-T_n)\right\|
\sum\limits_{k=1}^m\begin{pmatrix}m\\k\end{pmatrix}
\left|\beta_n^k\right|\left\|T_n^{m-k}\right\|\to0\quad\text{as $n\to\infty$,}
\end{multline*}
so
$
\left|\lambda\right|^{m+1}\left\|R_\lambda(T)(T-{A}_n){A}_n^m\right\|<\frac12
$
for all $n$ sufficiently large. Note that, for such $n$,
the right-hand side of equation \eqref{eq3.12} does represent an invertible 
operator on $L^2$. This makes the last factor
\begin{equation}\label{eq3.13}
I-\lambda{A}_n=(1-\beta_n\lambda)
\left(I-\tfrac{\lambda}{1-\beta_n\lambda}T_n\right)
\end{equation}
on the left-hand side one-to-one and so invertible, as $T_n$ is compact.
Hence, for such $n$, $\frac{\lambda}{1-\beta_n\lambda}\in\Pi(T_n)$,
$\lambda\in\Pi({A}_n)$, and 
\begin{gather*}
\left\|A_{n\mid\lambda}\right\|=
\frac1{|\lambda|}\left\|R_\lambda({A}_n)-I\right\|
\\=
\frac1{|\lambda|}\left\|
\left[I+\lambda^{m+1}R_\lambda(T)(T-{A}_n){A}_n^m\right]^{-1}
R_\lambda(T)\left((1-\lambda T)\sum\limits_{k=0}^{m-1}(\lambda{A}_n)^k+
(\lambda{A}_n)^m\right)-I
\right\|
\\\leqslant\frac1{|\lambda|}\frac{ \left\|R_\lambda(T)\right\|(1+\left|\lambda\right|\left\|T\right\|)
\sum\limits_{k=0}^{m}\left|\lambda\right|^{k} \left\|{A}_n\right\|^k }
{1-\left|\lambda\right|^{m+1}\left\|R_\lambda(T)(T-{A}_n){A}_n^m\right\|}
+\frac1{|\lambda|}
\\\leqslant M(\lambda):=
\frac2{|\lambda|}\left\|R_\lambda(T)\right\|(1+\left|\lambda\right|\left\|T\right\|)
\sum\limits_{k=0}^{m}\left|\lambda\right|^{k} (\max_{n\in\mathbb{N}}|\beta_n|+\|T\|)^k
+\frac1{|\lambda|},
\end{gather*}
where in the second equality use has been made of equation \eqref{eq3.12}. 
Thus (see \eqref{eq3.2}), $\lambda\in\nabla_{\mathfrak{b}}(\{A_n\})$, and
\eqref{eq3.11} now follows by \eqref{eq3.4new}. The theorem is proved.
\end{proof}
\begin{remark}
Observe by \eqref{TnTn} that \eqref{eq3.9} implies
\begin{equation}\label{eq3.10}
\left\|(T-\widetilde{T}_n)\widetilde{T}_n^m\right\|
=\left\|(T-T_n)T_n^mP_n\right\|\to 0
\quad\text{as $n\to\infty$}.
\end{equation}
The same result as in the above theorem is obtained, similarly,
with the sequence $\{\widetilde{T}_n\}$ satisfying \eqref{eq3.10}
and it reads as follows:
\begin{equation*}
\nabla_{\mathfrak{s}}(\{\beta_nI+\widetilde{T}_n\})=\overset{\circ}\Pi(T)\subset
\nabla_{\mathfrak{b}}(\{\beta_nI+\widetilde{T}_n\})
\end{equation*}
for any complex null sequence $\{\beta_n\}$.
Consequently, under condition \eqref{eq3.9},
\begin{equation}\label{eq3.14}
\nabla_{\mathfrak{s}}(\{\beta_nI+T_n\})=
\nabla_{\mathfrak{s}}(\{\beta_nI+\widetilde{T}_n\})=\overset{\circ}\Pi(T)
\end{equation}
for any complex null sequence $\{\beta_n\}$.
\end{remark}
The following two corollaries may be of interest for further applications.
%%%%%%%%%%%%%%%%%%%%%%%%%%%%%%%%%%%%%%%%%%%%%%%%%%%%%%%%%%%%%%%%%%%%%%%%%
\begin{corollary}
If $\lambda\in\overset{\circ}\Pi(T)$ and condition \eqref{eq3.9} holds, then  $\lambda$ is
not the limit of any sequence $\left\{\xi_n\right\}$ satisfying
$\xi_n\in\Lambda(T_n)$ at each $n$.
\end{corollary}
%%%%%%%%%%%%%%%%%%%%%%%%%%%%%%%%%%%%%%%%%%%%%%%%%%%%%%%%%%%%%%%%%%%%
\begin{proof}[\indent Proof]
Assume, on the contrary, that there exists a sequence $\left\{\xi_n\right\}$ with
$\xi_n\in\Lambda(T_n)$ such that  $\xi_n\to\lambda\in\overset{\circ}\Pi(T)$
as $n\to\infty$.
Theorem~\ref{thmeq3.4} says that $\lambda\in\Pi(\beta_nI+T_n)$
for all $n$ sufficiently large, where
$\beta_n=\frac{\xi_n-\lambda}{\lambda\xi_n}$.
This implies via \eqref{eq3.13} that
$\frac{\lambda}{1-\beta_n\lambda}=\xi_n\in\Pi(T_n)$
for all sufficiently large $n$, which is a contradiction. The corollary is proved.
\end{proof}
\begin{corollary}
If an operator $T$ with condition \eqref{eq3.9} is self-adjoint (that is, such that $T^*=T$) then
\begin{equation}\label{eq3.15}
\nabla_{\mathfrak{b}}(\{T_n\})=\nabla_{\mathfrak{b}}(\{\widetilde{T}_n\})
=\nabla_{\mathfrak{s}}(\{\widetilde{T}_n\})=
\nabla_{\mathfrak{s}}(\{T_n\})=\overset{\circ}\Pi(T).
\end{equation}
\end{corollary}
\begin{proof}[\indent Proof]
From the facts proved above it follows that
$
\overset{\circ}\Pi(T)\subset \nabla_{\mathfrak{b}}(\{T_n\})
\subset\nabla_{\mathfrak{b}}(\{\widetilde{T}_n\}).
$
This and \eqref{eq3.14} together show that to prove \eqref{eq3.15} it is
enough to prove that
$\nabla_{\mathfrak{b}}(\{\widetilde{T}_n\})\subset\overset{\circ}\Pi(T)$.
Suppose $\lambda\in\nabla_{\mathfrak{b}}(\{\widetilde{T}_n\})$, so
there is a positive constant $M$ such that
\begin{equation}\label{eq3.16}
\left\|\widetilde{T}_{n\mid\lambda}\right\|\leqslant M
\end{equation}
for all sufficiently large $n$,
but suppose, contrary to $\lambda\in\overset{\circ}\Pi(T)$, that $\lambda\in\Lambda(T)$.
Then, by  Theorem~VIII.24 of \cite[p.~290]{Reed:Sim1},
there exists a sequence $\lambda_n\in\Lambda(\widetilde{T}_n)$ $(n\in\mathbb{N})$ such that
$\lambda_n\to\lambda$ as $n\to\infty$. Consequently,
$$
\left|\lambda\right|\left\|\widetilde{T}_{n\mid\lambda}\right\|+1\geqslant
\left\|R_\lambda(\widetilde{T}_n)\right\|\geqslant\dfrac{\left|\lambda_n\right|}
{\left|\lambda_n-\lambda\right|}\to+\infty,
$$
which, however, is incompatible with \eqref{eq3.16}.
The corollary is proved.
\end{proof}
\begin{remark}
In terms of kernels, condition \eqref{eq3.9} (resp., \eqref{eq3.10}) means
that nuclear operators, induced on $L^2$ by the (explicit) kernels
\begin{gather*}
\boldsymbol{J}_n(s,t)=
\widehat{\chi}_n(s)\int\limits_{\mathbb{I}_n}\boldsymbol{T}(s,x)\boldsymbol{T}^{[m]}_n(x,t)\,dx\\
(\text{resp.\ }
\boldsymbol{\widetilde{J}}_n(s,t)=
\widehat{\chi}_n(s)
\int\limits_{\mathbb{I}_n}\boldsymbol{T}(s,x)\boldsymbol{\widetilde{T}}_n^{[m]}(x,t)\,dx),
\end{gather*}
have their operator norm going to $0$ as $n$ goes to infinity. In particular, if
the nuclear operators $(I-P_n)TP_n$, with kernels
$\widehat{\chi}_n(s)\boldsymbol{T}(s,t)\chi_n(t)$,
converge to zero operator in the operator norm as $n\to\infty$, then both
conditions \eqref{eq3.9} and \eqref{eq3.10} automatically hold with any fixed
$m$ in $\mathbb{N}$, and
this may happen even if $T$ is not a compact operator. Note, incidentally, that 
for the latter there is a stronger conclusion about the uniform-on-compacta
convergence on all of $\overset{\circ}\Pi(T)$:
\end{remark}
\begin{theorem}
If $T$ is a compact operator, then the following limits hold:
\begin{equation}\label{eq5.1} 
\begin{gathered}
\lim_{n\to\infty}
\sup_{\lambda\in{\mathfrak{K}}}\left\|\boldsymbol{{t}}_{n\mid\lambda}^{\boldsymbol{\prime}}-\boldsymbol{t}^{\boldsymbol{\prime}}_{\mid\lambda}\right\|_{C\left(\mathbb{R},L^2\right)}=0,
\quad
\lim_{n\to\infty}\sup_{\lambda\in{\mathfrak{K}}}\left\|\boldsymbol{{t}}_{n\mid\lambda}-
\boldsymbol{t}_{\mid\lambda}\right\|_{C\left(\mathbb{R},L^2\right)}=0,
\\
\lim_{n\to\infty}\sup_{\lambda\in{\mathfrak{K}}}\left\|\boldsymbol{{T}}_{n\mid\lambda}-
\boldsymbol{T}_{\mid\lambda}\right\|_{C\left(\mathbb{R}^2,\mathbb{C}\right)}=0
\end{gathered}
\end{equation}
for any choice of a compact subset $\mathfrak{K}$ of $\overset{\circ}\Pi(T)$ (compare with \eqref{eq3.3d}-\eqref{eq3.7d}).
\end{theorem}
\begin{proof}[\indent Proof]
Let $\mathfrak{K}$ be a compact subset of $\overset{\circ}\Pi(T)$.
Since under the stated hypotheses on $T$
\begin{equation}\label{eq5.2}
\lim_{n\to\infty}\left\|{T}_n-T\right\|=0\quad\text{and}\quad
\lim_{n\to\infty}\left\|\widetilde{T}_n-T\right\|=0,
\end{equation}
it follows
from \eqref{eqd3.16}, \eqref{eq3.22},
and Theorem~\ref{thmeq3.4} (all applied with $\beta_n$ all taken equal to zero)
that for some positive constant $M$
\begin{equation}\label{eq5.3}
\sup_{\lambda\in\mathfrak{K}}\left\|\widetilde{T}_{n\mid\lambda}\right\|
+\sup_{\lambda\in\mathfrak{K}}\left\|T_{n\mid\lambda}\right\|\leqslant M
\end{equation}
for all sufficiently large $n$.
Transforming the Fredholm resolvent differences $\widetilde{T}_{n\mid\lambda}-T_{\mid\lambda}$ 
and ${T}_{n\mid\lambda}-T_{\mid\lambda}$
into products of operators via the second resolvent equation \eqref{secFResEq} 
and subsequently using \eqref{eq5.2} and \eqref{eq5.3} then leads to the limit-relations:
\begin{equation}\label{eq5.4}
\lim_{n\to\infty}\sup_{\lambda\in\mathfrak{K}}
\left\|\widetilde{T}_{n\mid\lambda}-T_{\mid\lambda}\right\|=0,\quad
\lim_{n\to\infty}\sup_{\lambda\in\mathfrak{K}}
\left\|{T}_{n\mid\lambda}-T_{\mid\lambda}\right\|=0.
\end{equation}
For $\beta_n=0$, proceeding the inequalities \eqref{eq3.27}-\eqref{eqGamb}
yields, respectively, the following estimates
\begin{multline*}
\sup_{\lambda\in{\mathfrak{K}}}\left\|\boldsymbol{{t}}_{n\mid\lambda}^{\boldsymbol{\prime}}-
\boldsymbol{t}^{\boldsymbol{\prime}}_{\mid\lambda}\right\|_{C\left(\mathbb{R},L^2\right)}
\leqslant
\left\|(P_n-I)\left(\boldsymbol{t}^{\boldsymbol{\prime}}(t)\right)\right\|_{C\left(\mathbb{R},L^2\right)}
\\
+
\|\boldsymbol{\tau}^{\boldsymbol{\prime}}\|_{C(\mathbb{R},\mathbb{R})}
\sup_{\lambda\in{\mathfrak{K}}}
\left(|\lambda|\left\|\widetilde{T}_{n\mid\lambda}-T_{\mid\lambda}\right\|\right),
%%%%%%%%%%%%%%%%%%%%%%%%%%%%%%%%
\\
\sup_{\lambda\in{\mathfrak{K}}}\left\|\boldsymbol{t}_{n\mid\lambda}-
\boldsymbol{t}_{\mid\lambda}\right\|_{C\left(\mathbb{R},L^2\right)}
\leqslant
\|\widehat{\chi}_n\boldsymbol{\tau}\|_{C(\mathbb{R},\mathbb{R})}
\sup_{\lambda\in{\mathfrak{K}}}\left(1+|\lambda|\left\|T_{\mid\lambda}\right\|\right)
\\
+
\|\boldsymbol{\tau}\|_{C(\mathbb{R},\mathbb{R})}
\sup_{\lambda\in{\mathfrak{K}}}\left(|\lambda|\left\|
{T}_{n\mid\lambda}-T_{\mid\lambda}\right\|\right),
%%%%%%%%%%%%%%%%%%%%%%%%%%%%%%%%%%%%%%%%%%
\\
\sup_{\lambda\in{\mathfrak{K}}}\left\|\boldsymbol{{T}}_{n\mid\lambda}-
\boldsymbol{T}_{\mid\lambda}\right\|_{C\left(\mathbb{R}^2,\mathbb{C}\right)}
\leqslant
\|\boldsymbol{\tau}\|_{C(\mathbb{R},\mathbb{R})}
\sup_{\lambda\in{\mathfrak{K}}}\left\|\boldsymbol{{t}}_{n\mid\lambda}^{\boldsymbol{\prime}}-
\boldsymbol{t}^{\boldsymbol{\prime}}_{\mid\lambda}\right\|_{C\left(\mathbb{R},L^2\right)}
\\
+\|\widehat{\chi}_n\boldsymbol{\tau}\|_{C(\mathbb{R},\mathbb{R})}
\|\boldsymbol{\tau}^{\boldsymbol{\prime}}\|_{C(\mathbb{R},\mathbb{R})}
\sup_{\lambda\in{\mathfrak{K}}}\left(|\lambda|\left\|R_\lambda(T)\right\|\right)
+
\|\boldsymbol{{T}}_n-\boldsymbol{T}\|_{C\left(\mathbb{R}^2,\mathbb{C}\right)},
\end{multline*}
whence the limits in \eqref{eq5.1} all hold by virtue of \eqref{eqnormkf},
\eqref{Katolem}, \eqref{ke1.3.1}, and \eqref{eq5.4}.
The theorem is proved.
\end{proof}


\begin{thebibliography}{99}

\bibitem{Akh}
\by     N.\,I.~Akhiezer
\paper  Integral operators with Carleman kernels
\jour   Uspekhi Mat. Nauk.
\vol    2
\issue  5
\pages  93--132
\yr     1947
\lang in Russian


\bibitem{Akh:Glaz}
\by       N.\,I.~Akhiezer, I.\,M.~Glazman
\book     Theory of linear operators in Hilbert space
\publaddr Moscow
\publ     Nauka
\yr       1966
\lang     in Russian



\bibitem{Buescu1}
\by     J. Buescu
\paper  Positive integral operators in unbounded domains
\jour   J. Math. Anal. Appl.
\vol    296
\issue  1
\pages  244--255 
\yr     2004


\bibitem{Carl:pap}
\by      T.~Carleman
\paper   Zur {T}heorie der linearen {I}ntegralgleichungen
\jour    Math. Z.
\vol     9
\pages   196--217
\yr      1921

\bibitem{Carl:book}
\by        T.~Carleman 
\book      Sur les \'equations int\'egrales singuli\`eres \`a noyau r\'eel et sym\'etrique
\publaddr  Uppsala
\publ      A.-B. Lundequistska Bokhandeln 
\yr        1923


\bibitem{Costley}
\by    C.\,G.~Costley
\paper On singular normal linear equations 
\jour  Can. Math. Bull. 
\vol   13 
\pages 199--203
\yr    1970


\bibitem{Halmos:Sun}
\by        P.~Halmos, V.~Sunder
\book      Bounded integral operators on $L^2$ spaces
\publaddr  Berlin
\publ      Springer
\yr        1978


\bibitem{Hille}
\by       E.~Hille
\book     Functional analysis and semi-groups
\publaddr New York
\publ     Amer. Math. Soc. Colloq. Publ. 31
\yr       1948


\bibitem{KanAk}
\by       L.\,V.~Kantorovich, G.\,P.~Akilov
\book     Functional analysis
\publaddr Moscow
\publ     Nauka
\yr       1984
\edition  3rd ed.
\lang     in Russian

\bibitem{Kato:book}
\by       T.~Kato
\book     Perturbation theory for linear operators
\publaddr Berlin-Heidelberg-New York
\publ     Springer-Verlag 
\yr       1980
\edition Corr. print. of the 2nd ed. 

\bibitem{Kor:book1}
\by        V.\,B.~Korotkov
\book      Integral operators
\publaddr  Novosibirsk
\publ      Nauka
\yr        1983
\lang      in Russian

\bibitem{Kor:problems}
\by      V.\,B.~Korotkov
\paper   Some unsolved problems of the theory of integral
         operators
\jour    Siberian Adv. Math.
\vol     7
\yr      1997
\pages   5--17


\bibitem{Kor:nonint1}
\by     V.\,B.~Korotkov
\paper  On the nonintegrability property of the Fredholm resolvent of some
        integral operators
\jour   Sibirsk. Mat. Zh.
\vol    39
\yr     1998
\pages  905--907
\lang   in Russian

\bibitem{Kor:alg}
\by         V.\,B.~Korotkov
\book       Introduction to the algebraic theory of integral operators
\publ       Far-Eastern Branch of the Russian Academy of Sciences
\publaddr   Vladivostok
\yr         2000
\lang       in Russian
\isbn       5-1442-0827-5
\totalpages 79

\bibitem{Misra}
\by    B.~Misra, D.~Speiser, G.~Targonski
\paper Integral operators in the theory of scattering
\jour  Helv. Phys. Acta
\vol   36
\pages 963--980
\yr    1963

\bibitem{Mih:pap}
\by      S.\,G.~Mikhlin
\paper   On the convergence of Fredholm series
\jour    Doklady AN SSSR
\vol     XLII
\issue   9
\yr      1944
\pages   374--377
\lang    in Russian


\bibitem{nov:Lon}
\by    I.\,M.~Novitskii
\paper Integral representations of linear operators by smooth Carleman kernels of Mercer type
\jour  Proc. Lond. Math. Soc. (3)
\vol   68
\issue 1
\pages 161--177
\yr    1994


\bibitem{Reed:Sim1}
\by      M.~Reed, B.~Simon
\book    Methods of modern mathematical physics. {I}. Functional analysis
\edition rev.
\publaddr San Diego 
\publ    Academic Press
\yr      1980

\bibitem{Smith:paper}
\by     F.~Smithies
\paper  The {F}redholm theory of integral equations
\jour   Duke Math.~J.
\vol    8
\pages  107--130
\yr     1941

\bibitem{Trji}
\by    W.\,J.~Trjitzinsky
\paper Singular integral equations with complex valued kernels
\jour  Ann. Mat. Pura Appl.
\vol   4
\issue 25
\pages 197--254
\yr    1946

\bibitem{Neu}
\by       J.~von Neumann
\book     Charakterisierung des Spektrums eines Integraloperators
\serial   Actual. scient. et industr.
\publaddr Paris
\publ     Hermann
\yr       1935
\vol      229
\lang     in German

\bibitem{Will}
\by    J.\,W. Williams
\paper Linear integral equations with singular normal kernels of class I
\jour  J. Math. Anal. Appl.
\vol   68
\issue 2
\pages 567--579
\yr    1979

\bibitem{Zaanen}
\by    A.\,C.~Zaanen
\paper An extension of Mercer's theorem on continuous kernels of positive type
\jour  Simon Stevin
\vol   29
\pages 113--124
\yr    1952

\end{thebibliography}
\end{document}